\documentclass[12pt,reqno]{amsart}
\usepackage{amsmath,amssymb,amsfonts,amsthm}
\usepackage[left=3cm, right=3cm, top=2.8cm, bottom=2.8cm]{geometry}
\usepackage[utf8]{inputenc}
\usepackage[T1]{fontenc}
\usepackage{color}
\usepackage{multirow}
\usepackage{graphicx}
\usepackage{hyperref}
\usepackage{graphicx}
\usepackage{calc}%
\usepackage{hyperref}
\hypersetup{
    colorlinks=true,
    linkcolor=blue,
    filecolor=magenta,
    urlcolor=cyan,
}

\DeclareMathOperator*{\esssup}{ess\,sup}
\newtheorem{theorem}{Theorem}
\newtheorem{lemma}[theorem]{Lemma}

\newtheorem{corollary}[theorem]{Corollary}
\newtheorem{definition}[theorem]{Definition}
\newtheorem{remark}[theorem]{Remark}

\theoremstyle{plain}

\begin{document}

\title[The Riemann-Liouville fractional integral]{The Riemann-Liouville fractional integral in Bochner-Lebesgue spaces I}


\author[P. M. Carvalho-Neto]{Paulo M. Carvalho-Neto}
\address[Paulo M. de Carvalho Neto]{Departamento de Matem\'atica, Centro de Ci\^{e}ncias F\'{\i}sicas e Matem\'{a}ticas, Universidade Federal de Santa Catarina, Florian\'{o}polis - SC, Brazil}
\email[]{paulo.carvalho@ufsc.br}
\author[R. Fehlberg J\'{u}nior]{Renato Fehlberg J\'{u}nior}
\address[Renato Fehlberg J\'{u}nior]{Departamento de Matem\'atica, Universidade Federal do Esp\'{\i}rito Santo, Vit\'{o}ria - ES, Brazil}
\email[]{fjrenato@yahoo.com.br}


\subjclass[2010]{26A33, 47G10, 46B50}


\keywords{Riemann-Liouville fractional integral, Bochner-Lebesgue space, operator theory, semigroup theory.}


\begin{abstract}
In this paper we study the Riemann-Liouville fractional integral of order $\alpha>0$ as a linear operator from $L^p(I,X)$ into itself, when $1\leq p\leq \infty$, $I=[t_0,t_1]$ (or $I=[t_0,\infty)$) and $X$ is a Banach space. In particular, when $I=[t_0,t_1]$, we obtain necessary and sufficient conditions to ensure its compactness. We also prove that Riemann-Liouville fractional integral defines a $C_0-$semigroup but does not defines a uniformly continuous semigroup. We close this study by presenting lower and higher bounds to the norm of this operator.
\end{abstract}

\maketitle

\section{Introduction}

From a historical point of view, the fundamental structures of fractional calculus have its origin in an inquiry about the possibility to extend the notion of derivative. In fact, it was Leibniz the first to introduce the idea of a symbolic notation to the nth derivative, namely $d^n/dx^n=D^n$, where $n$ was considered a positive integer. It was L'Hospital that in 1695 directly questioned Leibniz about the possibility to consider $n=1/2$ (see \cite{Leib1,Ro2} for details). Perhaps, like B. Ross pointed out in \cite{Ro1}, ``it was naive play with symbols that prompted L'Hospital to ask Leibniz about the possibility that $n$ be a fraction'', however this questioning was the responsible for the consolidation of this theory in mathematics.

After its conception as an area of expertise, the fractional calculus have gone through a huge paradigm shift, going from being a forgotten subject to a trending one; today we may cite several papers that have been published in important mathematical journals: see \cite{AlCaVa1,CarFe1,Ch1,DiKeSiZa1, DoKi1,DoKi2, GiNa1, KiLiLiYa1,KiKiLi1,NaRy1,Po1} as a sample of this literature.

Nevertheless, besides the several open questions regarding this theory, this work intends to give attention to a more systematic study of the fractional integral itself. To be more precise, let us describe the highlights of each section of this manuscript.

In Section \ref{sec2} we introduce the theory of Bochner measurable and integrable functions, and present the Riemann-Lioville fractional integral of order $\alpha>0$. Then we give necessary and sufficient conditions to ensure the existence of Riemann-Liouville fractional integral of a Bochner integrable function.

In Section \ref{sec3} we prove our first main result, which presents necessary and sufficient conditions for the compactness of the Riemann-Liouville fractional integral of order $\alpha>0$, as an operator from $L^p(I;X)$ into itself. Besides that, we also study the continuity with respect to the order of integration. We end this section using the theory of semigroups of linear operators to prove the absence of the uniform continuity when $\alpha=0$.

We dedicate Section \ref{sec4} to present lower and higher bounds to this norm, proving that when $p=1$ or $p=\infty$, we can obtain an exact value.

Since there is no unifying source that contains the classical results we use in this manuscript, we dedicate Section \ref{sec5} to summarize them.

\section{The Riemann-Liouville fractional integral}
\label{sec2}

In order to study the Riemann-Liouville fractional integral in a more general context, we first present a few classical definitions and results about the Bochner measurable and integrable functions. We recall that vector-valued functions from a measurable domain of $\mathbb{R}^n$ into a Banach space, is a topic widely studied in the literature; some good references are W. Arendt et al. \cite{ArBaHiNe1}, U. Diestel \cite{Di1}, J. Mikusi\'nski \cite{Mik1}, J. Pettis \cite{Pe1} and R. E. Showalter \cite{Sho1}.

It worths to emphasize that throughout this manuscript we assume that $t_0<t_1$ are fixed real numbers, $I$ is a non-empty subset of $\mathbb{R}$ and $X$ always denotes a Banach vector space.

\begin{definition}\label{firstdef} Let $n\in\mathbb{N}^*:=\mathbb{N}\setminus{\{0\}}$, $S$ a subset of $\mathbb{R}^n$, $\Sigma$ a $\sigma-$algebra on the set $S$ and $\mu$ a measure in $(S,\Sigma)$. We call the triplet $(S,\Sigma,\mu)$ a measure space.\vspace*{0.2cm}
\begin{itemize}
\item[(i)] A step function $f:S\rightarrow X$ (i.e., a finite-valued function) is Bochner measurable if $f^{-1}(\{x\})\in \Sigma$, for every $x\in X$. If it also holds that $\mu\big(f^{-1}(\{x\})\big)<\infty$, for every $x\not=0$, we say that $f$ is integrable in $S$ and we define
    $$\int_S{f}\,d\mu=\sum_{x\in X}x\mu\big(f^{-1}(x)\big).$$\vspace*{0.1cm}
\item[(ii)] A function $f:S\rightarrow X$ is said to be Bochner measurable if there is a sequence $\{f_n(s)\}_{n=1}^\infty$ of Bochner measurable step functions for which $f_n(s)\rightarrow f(s)$ in the topology of $X$, for almost every $s\in S$.\vspace*{0.2cm}
\item[(iii)] A function $f:S\rightarrow X$ is Bochner integrable if there is a sequence $\{f_n(s)\}_{n=1}^\infty$ of integrable step functions such that
    $$\lim_{n\rightarrow\infty}{\int_S{\big\|f_n(s)-f(s)\big\|_X}\,d\mu}=0$$
    and each integrand belongs to $L^1(S;\mu)$. In this case, $\int_S{f_n(s)}\,d\mu$ converges in $X$ to a limit that is independent of the sequence of step functions, which we denote by $\int_S{f(s)}\,d\mu$. \vspace*{0.2cm}
\end{itemize}
\end{definition}

\begin{remark}\label{remarknova} We highlight two important results that arise from definition \ref{firstdef}.
\begin{itemize}
\item[(i)] Assume that $I\subset\mathbb{R}$. If $f:I\rightarrow X$ is a Bochner measurable function and $g:\mathbb{R}\rightarrow \mathbb{R}$ is a Lebesgue measurable function, then we deduce that function $\mathbb{R}\times I\ni (t,s)\mapsto g(t-s)f(s)$ is Bochner measurable (see \cite[Chapter XIV - Lemma 1.1]{Mik1}). Just note that  there are sequences of step functions  $\{f_n(s)\}_{n=1}^\infty$ and  $\{g_n(s)\}_{n=1}^\infty$ such that $f_n(s)\rightarrow f(s)$ (in the topology of $X$), for almost every $s\in I$, and $g_n(t-s)\rightarrow g(t-s)$ (in the topology of $\mathbb{R}$), for almost every $t,s\in \mathbb{R}$, thus
    \begin{multline*}\hspace*{1cm}\big\|g_n(t-s)f_n(s)-g(t-s)f(s)\big\|_X\leq |g_n(t-s)-g(t-s)|\big\|f_n(s)\big\|_X\\+|g(t-s)|\big\|f_n(s)-f(s)\big\|_X.\end{multline*}
    Since $\{f_n(s)\}_{n=1}^\infty$ is limited, for each fixed $s\in I$, the result follows.\vspace*{0.2cm}
\item[(ii)] The second result, proved by Bochner (see \cite[Theorem II.2.2]{Di1} or \cite[Theorem 1.2.1]{Vr1}), connects  part $(ii)$ and $(iii)$ from definition \ref{firstdef} and it allows us to introduce the Bochner-Lebesgue spaces. This result can be stated as: a function $f:I\rightarrow X$ is Bochner integrable if, and only if, $f$ is Bochner measurable and $\|f\|_X\in L^1(I;\mu)$.
\end{itemize}
\end{remark}

From now on we always consider $I\subset\mathbb{R}$ with the induced Lebesgue measure from $\mathbb{R}$. Thus we are able to introduce the classical Bochner-Lebesgue spaces $L^{p}(I;{X})$.

\begin{definition} Consider $1\leq p\leq\infty$.
We use the symbol $L^{p}(I;{X})$ to represent the set of all Bochner measurable functions $f:I\rightarrow X$ for which $\|f\|_X\in L^{p}(I;\mathbb{R})$, where $L^{p}(I;\mathbb{R})$ stands for the classical Lebesgue space. Moreover, $L^{p}(I;{X})$ is a Banach space when considered with the norm
$$\|f\|_{L^p(I;X)}:=\left\{\begin{array}{ll}\bigg[\displaystyle\int_I{\|f(s)\|^p_X}\,ds\bigg]^{1/p},&\textrm{ if }p\in[1,\infty),\vspace*{0.3cm}\\
\esssup_{s\in I}\|f(s)\|_X,&\textrm{ if }p=\infty.\end{array}\right.$$
When $I=[t_0,t_1]$  (or $I=[t_0,\infty)$) almost everywhere, we shall also use the notation $L^p(t_0,t_1;X)$ (or $L^p(t_0,\infty;X)$).\vspace*{0.2cm}
\end{definition}

\begin{remark}\label{remarktestf}
\begin{itemize}
\item[(i)] Several geometric properties of $X$ are lifted to $L^p(I;X)$ (see the paper of Dowling-Hu-Mupasir \cite{DoHuMu1} or the survey paper of Smith \cite{Sm1}). Specifically, it is well-known that strict convexity, uniform convexity, local uniform convexity, midpoint local uniform convexity and reflexivity are all lifted from $X$ to $L^p(I;X)$, when $1 < p < \infty$.
\item[(ii)] We shall also emphasize the existence of test functions in the context of Bochner-Lebesgue spaces. In fact, if $I$ is an open set, we shall call the set of every infinitely differentiable function $\phi:I\rightarrow X$ with compact support by $C^\infty_c(I;X)$, and call each of its elements test functions. It also holds that $C^\infty_c(I;X)$ is a dense subset of $L^p(I;X)$, when $1\leq p\leq\infty$. For more details on this subject, we may refer to \cite{ArBaHiNe1,Scw1,Ze1}.
\end{itemize}
\end{remark}

\subsection{Some Results About RL Fractional Integral}

Now we present the classical fractional integral used in the Bochner-Lebesgue spaces. Henceforth, we assume that $I\subset\mathbb{R}$ always denotes $[t_0,t_1]$ or $[t_0,\infty)$.

\begin{definition}\label{riemannint} For $\alpha\in(0,\infty)$ and $f:I\rightarrow{X}$, the Riemann-Liouville (RL for short) fractional integral of order $\alpha$ at $t_0$ of function $f(t)$ is defined by
\begin{equation}\label{fracinit}J_{t_0,t}^\alpha f(t):=\dfrac{1}{\Gamma(\alpha)}\displaystyle\int_{t_0}^{t}{(t-s)^{\alpha-1}f(s)}\,ds,\end{equation}
for every $t\in I$ such that integral \eqref{fracinit} exists. Above $\Gamma(z)$ denotes the classical Euler's gamma function.
\end{definition}

\begin{remark}
Recall that for $k\in\mathbb{N}$ and $f\in L^1(I;X)$, it holds
$$J_{t_0,t}^k f(t)=\dfrac{1}{(k-1)!}\displaystyle\int_{t_0}^{t}{(t-s)^{k-1}f(s)}\,ds,$$
for every $t\in I$. The above equality is called Cauchy formula for $k-$repeated integration of function $f$. This identity is the notion that induces Definition \ref{riemannint}.
\end{remark}

The next result binds the existence of $J^\alpha_{t_0,t}f(t)$ with the existence of $J^\alpha_{t_0,t}\|f(t)\|_X$.

\begin{lemma}\label{interlemma} Consider $\alpha>0$ and $f:[t_0,t_1]\rightarrow X$ a Bochner measurable function. Then $J^\alpha_{t_0,t}f(t)$ exists for almost every $t\in[t_0,t_1]$ if and only if $J^\alpha_{t_0,t}\|f(t)\|_X$ exists for almost every $t\in[t_0,t_1]$.\vspace*{0.2cm}
\end{lemma}

\begin{proof} If we assume that $J_{t_0,t}f(t)$ exists for almost every $t\in[t_0,t_1]$, then function $s\mapsto (t-s)^{\alpha-1}f(s)$ is a Bochner integrable function in $[t_0,t]$, for almost every $t\in[t_0,t_1]$.
But then, $s\mapsto (t-s)^{\alpha-1}f(s)$ is the limit of integrable step functions $\{g_{n,t}(s)\}_{n=1}^\infty$ that satisfies
        $$\lim_{n\rightarrow\infty}\int_{t_0}^{t}\big\|g_{n,t}(s)-(t-s)^{\alpha-1}f(s)\big\|_X\,ds=0.$$
        Hence, $\{\|g_{n,t}(s)\|_X\}_{n=1}^\infty$ are integrable step (real) functions that, by left triangular inequality, satisfies
        \begin{multline*}\hspace{1.8cm}\lim_{n\rightarrow\infty}\int_{t_0}^{t_1}\Big|\|g_{n,t}(s)\|_X-(t-s)^{\alpha-1}\|f(s)\|_X\Big|\,ds\\\leq \lim_{n\rightarrow\infty}\int_{t_0}^{t_1}\big\|g_{n,t}(s)-(t-s)^{\alpha-1}f(s)\big\|_X\,ds=0.\end{multline*}
        Therefore, $s\mapsto(t-s)^{\alpha-1}\|f(s)\|_X$ is an integrable function in $[t_0,t]$, for almost every $t\in[t_0,t_1]$, what implies that $J_{t_0,t}\|f(t)\|_X$ exists for almost every $t\in[t_0,t_1]$.

      On the other hand, If we assume that $J_{t_0,t}\|f(t)\|_X$ exists for almost every $t\in[t_0,t_1]$, we would have
      $$\int_{t_0}^t\|(t-s)^{\alpha-1}f(s)\|_X\,ds$$
      exists for almost every $t\in[t_0,t_1]$, i.e., $s\mapsto \|(t-s)^{\alpha-1}f(s)\|_X$ belongs to $L^1(t_0,t;\mathbb{R})$, for almost every $t\in[t_0,t_1]$. Since item $(i)$ of Remark \ref{remarknova} ensures that $s\mapsto (t-s)^{\alpha-1}f(s)$ is a Bochner measurable function in $[t_0,t]$, for almost every $t\in[t_0,t_1]$, item $(ii)$ of Remark \ref{remarknova} allows us to conclude that $s\mapsto (t-s)^{\alpha-1}f(s)$ is a Bochner integrable function in $[t_0,t]$, for almost every $t\in[t_0,t_1]$. In other words, we have deduced that $J_{t_0,t}f(t)$ exists for almost every $t\in[t_0,t_1]$.
\end{proof}

Now we present a necessary and sufficient condition to ensure the existence of the RL fractional integral for almost every $t\in I$. We were slightly inspired by Martinez-Sanz-Martinez (see \cite[Lemma 2.1]{MaSaMa1} for details).

\begin{theorem}\label{ajustefinal} Let $\alpha>0$.
\begin{itemize}
  \item[(i)] $f\in L^1(t_0,b;X)$, for every $b\in(t_0,t_1)$, if and only if, $J^\alpha_{t_0,t}f(t)$ exists for almost every $t\in[t_0,t_1]$.\vspace*{0.2cm}
  \item[(ii)] If $f\in L^1(t_0,t_1;X)$, then $J^\alpha_{t_0,t}f(t)$ is Bochner integrable in $[t_0,t_1]$.
\end{itemize}
\end{theorem}

\begin{proof} $(i)$ Let $b\in(t_0,t_1)$ and $f\in L^1(t_0,b;X)$. Define function $\psi_b:[t_0,b]\times[t_0,b]\rightarrow X$ by
$$\psi_b(t,s)=g_\alpha(t-s)f(s),$$
where $g_\alpha:\mathbb{R}\rightarrow\mathbb{R}$ is given by
\begin{equation}\label{galphafunc}g_\alpha(s):=\left\{\begin{array}{cl} s^{\alpha-1}/\Gamma(\alpha),&s>0,\vspace{0.2cm}\\
0,&s\leq0.\end{array}\right.\end{equation}

Since $f(t)$ is a Bochner measurable function, item $(i)$ of Remark \ref{remarknova} ensures that $\psi_b(t,s)$ is a Bochner measurable function in $[t_0,b]\times[t_0,b]$. Now, observe by Holder's inequality that
\begin{multline*}\Gamma(\alpha+1)\int_{t_0}^{b}{\int_{t_0}^{b}{\|\psi_b(t,s)\|_X}\,dt}\,ds=\alpha\int_{t_0}^{b}\left[\int_{s}^{b}(t-s)^{\alpha-1}\|f(s)\|_X\,dt\right]ds\\
=\int_{t_0}^{b}{(b-s)^{\alpha}\|f(s)\|_X}\,ds\leq (t_1-t_0)^\alpha\|f\|_{L^1(t_0,b;X)}.
\end{multline*}

Thus, item $(ii)$ of Remark \ref{remarknova} allows us to conclude that $\psi_b(s,r)$ is Bochner integrable in $[t_0,b]\times[t_0,b]$. But then, Theorem \ref{FubiniFubini} allows us to conclude that function
$$[t_0,b]\ni t\mapsto \int_{t_0}^{t}{\psi(t,s)}\,ds\,\,\,\,\Big(=J_{t_0,t}^\alpha f(t)\Big)$$
exists almost every and is Bochner integrable in $[t_0,b]$. Since $b$ was arbitrarily chosen in $(t_0,t_1)$, we deduce that $J_{t_0,t}^\alpha f(t)$ exists for almost every $t\in[t_0,t_1]$.

Conversely, let $b\in(t_0,t_1)$. Since $J^\alpha_{t_0,t}\|f(t)\|_X$ exists for almost every $t\in[t_0,t_1]$ (by Lemma \ref{interlemma}), there exists $t^{*}\in(b,t_1)$ such that $s\mapsto(t^{*}-s)^{\alpha-1}\|f(s)\|_X$ belongs to $L^1(t_0,t^{*},X)$. Consequently, if $s\in [t_0,b]$ we deduce that
    \begin{multline*}\hspace*{1.9cm}\|f(s)\|_X=\Big[(t^{*}-s)^{\alpha-1}\|f(s)\|_X\Big](t^{*}-s)^{1-\alpha}\\\leq\left\{\begin{array}{ll}\Big[(t^{*}-s)^{\alpha-1}\|f(s)\|_X\Big](t^{*}-b)^{1-\alpha},&\textrm{if }\alpha>1,\vspace*{0.3cm}\\
    \Big[(t^{*}-s)^{\alpha-1}\|f(s)\|_X\Big](t^{*}-t_0)^{1-\alpha},&\textrm{if }0<\alpha\leq1.\end{array}\right.\end{multline*}
    Thus, we conclude that
    \begin{multline*}\hspace*{1.8cm}\int_{t_0}^b{\|f(s)\|_X}\,ds\leq{\Big[(t^{*}-t_0)^{1-\alpha}+(t^{*}-b)^{1-\alpha}\Big]}\int_{t_0}^b{(t^{*}-s)^{\alpha-1}\|f(s)\|_X}\,ds\\
    \leq\Gamma(\alpha)\big[(t^{*}-t_0)^{1-\alpha}+(t^{*}-b)^{1-\alpha}\big]\,J^\alpha_{t_0,t}\|f(t)\|_X\Big|_{t=t^{*}},\end{multline*}
    i.e., $f\in L^1(t_0,b;X)$. Since $b\in(t_0,t_1)$ was an arbitrarily chosen, we conclude that $f\in L^1(t_0,b;X)$ for every $b\in(t_0,t_1)$.

$(ii)$ Proceeding exactly as in the first part of the proof of item $(i)$, however considering $f\in L^1(t_0,t_1;X)$, we shall obtain that  $J^\alpha_{t_0,t}f(t)$ is Bochner integrable in $[t_0,t_1]$.
\end{proof}

As a direct consequence of Theorem \ref{ajustefinal}, we have

\begin{corollary} Let $\alpha>0$. It holds that $f\in L^1(t_0,b;X)$, for all $b>t_0$, if and only if, $J^\alpha_{t_0,t}f(t)$ exists for almost every $t\in[t_0,\infty)$.
\end{corollary}

\begin{remark}\label{remark5} The above results and definitions allow us to point out some important questions about the RL fractional integral of order $\alpha>0$.
\begin{itemize}
\item[(i)] If $f\in L^1(t_0,t_1;X)$ and, for $\alpha>0$ we consider function $g_\alpha:\mathbb{R}\rightarrow\mathbb{R}$ given by \eqref{galphafunc}, by defining $f_{t_0}:\mathbb{R}\rightarrow X$ as equal to $f$ in $[t_0,t_1]$ and equal to zero otherwise, we deduce that
$$J_{t_0,t}^\alpha f(t)=\big[\,g_{\alpha}*{f_{t_0}}\,\big](t),\quad \textrm{for almost every }t\in[t_0,t_1].$$
We omit the subscript $t_0$ in function $f(t)$ when it does not lead to any confusion. For more details on convolutions of Bochner integrable functions see \cite[Chapter XIV]{Mik1}.\vspace*{0.2cm}
\item[(ii)] For any $f\in L^p(t_0,t_1;\mathbb{R})$, with $1\leq p\leq\infty$, besides knowing that $J_{t_0,t}f(t)$ exists for almost every $t\in[t_0,t_1]$, we have that
\begin{equation*}\lim_{\alpha\rightarrow0^+}{\big\|J_{t_0,t}^\alpha f-f\big\|_{L^p(t_0,t_1;\mathbb{R})}}=0\end{equation*}
and that for any $\alpha_0>0$
\begin{equation*}\lim_{\alpha\rightarrow\alpha_0}\left\{\sup_{f\in L^p(t_0,t_1;\mathbb{R})\setminus\{0\}}\left[{\dfrac{\big\|J_{t_0,t}^\alpha f-J_{t_0,t}^{\alpha_0}f\big\|_{L^p(t_0,t_1;\mathbb{R})}}{\|f\|_{L^p(t_0,t_1;\mathbb{R})}}}\right]\right\}=0.\end{equation*}

Therefore, for the completeness of the definition of the RL fractional integral in the space $L^p(t_0,t_1;\mathbb{R})$, we define $J_{t_0,t}^0 f(t):=f(t),$ for every $f\in L^p(t_0,t_1;\mathbb{R})$. For details on the proof see \cite[Theorem 2.6]{SaKiMa1}. We emphasize that the case of Bochner-Lebesgue spaces is proved in Section \ref{sec3}. \vspace*{0.2cm}
\item[(iii)] We point out that the continuity in the order of integration of the RL fractional integral as an operator from $L^p(t_0,t_1;\mathbb{R})$ into itself is a well known result. To be more specific, it holds that
\begin{equation*}\|J_{t_0,t}^\alpha f(t)\|_{L^p(t_0,t_1;\mathbb{R})}\leq\dfrac{(t_1-t_0)^\alpha}{\Gamma(\alpha+1)}\|f\|_{L^p(t_0,t_1;\mathbb{R})},\end{equation*}
for every function $f\in L^p(t_0,t_1;\mathbb{R})$. This statement can be found in several classical texts (see for instance the proof of \cite[Theorem 2.6]{SaKiMa1}), however its demonstration, in general, is not displayed.
\end{itemize}
\end{remark}

\section{The compactness of RL fractional integral operator in Bochner-Lebesgue spaces}\label{sec3}

In this section we discuss the compactness of that RL fractional integral of order $\alpha>0$ as a liner operator from $L^p(I;X)$ into itself, for any $1\leq p\leq\infty$. However, since interval $I$ can be bounded or unbounded, we split this discussion in two cases.

\subsection{The Bounded Interval Case} Let us show first that the RL fractional integral operator of order $\alpha>0$, when $I=[t_0,t_1]$ almost everywhere, is a bounded operator. We emphasize that the boundness of this operator is widely known in the literature when $X=\mathbb{R}$ (see for instance \cite{SaKiMa1}) and that an adaptation to the case when $X\not=\mathbb{R}$ does not brings any new complexities. However, for the completeness of this paper, below we give a non classical proof to the case $X=\mathbb{R}$, by using interpolation, and then adapt it to the case when $X\not=\mathbb{R}$.

\begin{theorem}\label{minkowskiseq} Let $\alpha>0$, $1\leq p\leq\infty$ and $f\in L^p(t_0,t_1;X)$. Then $J_{t_0,t}^\alpha f(t)$ is Bochner integrable and belongs to $L^p(t_0,t_1;X)$. Furthermore, it holds that
\begin{equation}\label{naturalinclusion}\left[\int_{t_0}^{t_1}{\left\|J_{t_0,t}^\alpha f(t)\right\|^p_X}\,dt\right]^{1/p}\leq \left[\dfrac{(t_1-t_0)^\alpha}{\Gamma(\alpha+1)}\right] \|f\|_{L^p(t_0,t_1;X)}.\end{equation}
In other words, $J_{t_0,t}^\alpha$ is a bounded linear operator from $L^p(t_0,t_1;X)$ into itself.
\end{theorem}

\begin{proof} Item $(ii)$ of Theorem \ref{ajustefinal} ensures that $J_{t_0,t}^\alpha f(t)$ is Bochner integrable in $[t_0,t_1]$.
Let us now prove that if $f\in L^p(t_0,t_1;X)$ then $J_{t_0,t}^\alpha f\in L^p(t_0,t_1;X)$, for any value $1\leq p\leq\infty$.

Assume initially that $X=\mathbb{R}$. Since
$$\int_{t_0}^{t_1}\|J_{t_0,t}^\alpha f(t)\|_X\,dt\leq \int_{t_0}^{t_1}\left[{\int_{t_0}^t{{{(t-w)^{\alpha-1}}\|f(w)\|_X}}\,dw}\right]\,dt,$$
and Corollary \ref{FubiniFubini2} guarantees
\begin{multline*}\int_{t_0}^{t_1}\left[{\int_{t_0}^t{{{(t-w)^{\alpha-1}}\|f(w)\|_X}}\,dw}\right]\,dt=\int_{t_0}^{t_1}\left[{\int_{w}^{t_1}{{{(t-w)^{\alpha-1}}}}\,dt}\right]\|f(w)\|_X\,dw\\
=\dfrac{1}{\alpha}\int_{t_0}^{t_1}(t_1-w)^\alpha\|f(w)\|_X\,dw,\end{multline*}
Holder's inequality gives us
$$\int_{t_0}^{t_1}\|J_{t_0,t}^\alpha f(t)\|_X\,dt\leq \left[\dfrac{(t_1-t_0)^\alpha}{\Gamma(\alpha+1)}\right]\int_{t_0}^{t_1}\|f(w)\|_X\,dw,$$
i.e., $J_{t_0,t}^\alpha f\in L^1(t_0,t_1;\mathbb{R})$.

On the other hand, for any $f\in L^\infty(t_0,t_1;\mathbb{R})$, Holder's inequality ensures
$$\dfrac{1}{\Gamma(\alpha)}\int_{t_0}^t{{{(t-w)^{\alpha-1}}|f(w)|}}\,dw\leq \left[\dfrac{(t-t_0)^\alpha}{\Gamma(\alpha+1)}\right]\esssup_{w\in [t_0,t_1]}|f(w)|,$$
and therefore we obtain the estimate
$$\esssup_{t\in [t_0,t_1]}|J_{t_0,t}^\alpha f(t)|\leq\left[\dfrac{(t_1-t_0)^\alpha}{\Gamma(\alpha+1)}\right]\esssup_{w\in [t_0,t_1]}|f(w)|,$$
i.e., $J_{t_0,t}^\alpha f\in L^\infty(t_0,t_1;\mathbb{R})$.

Thus, if we consider Theorem \ref{riezthorin} with $p_1=1$, $p_2=\infty$ and $\theta\in(0,1)$, for every function $f\in L^{{1/(1-\theta)}}(t_0,t_1;\mathbb{R})$ we have
$$\|J^\alpha _{t_0,t}f\|_{L^{1/(1-\theta)}(t_0,t_1;\mathbb{R})} \leq\left[\dfrac{(t_1-t_0)^\alpha}{\Gamma(\alpha+1)}\right]\|f\|_{L^{1/(1-\theta)}(t_0,t_1;\mathbb{R})}.$$
i.e., $J_{t_0,t}^\alpha f\in L^{{1/(1-\theta)}}(t_0,t_1;\mathbb{R})$. The fact that the range of $1/(1-\theta)$ is $(1,\infty)$ completes the proof of this theorem for this particular case.

Now assume that $f\in L^p(t_0,t_1;X)$, for any $1\leq p\leq\infty$. Since it holds that function $\|f(t)\|_X$ belongs to $L^p(t_0,t_1;\mathbb{R})$, by applying the already proved version of this result, we deduce that
$$\|J^\alpha _{t_0,t}\|f\|_X\|_{L^{p}(t_0,t_1;\mathbb{R})}\leq\left[\dfrac{(t_1-t_0)^\alpha}{\Gamma(\alpha+1)}\right]\|\|f\|_X\|_{L^{p}(t_0,t_1;\mathbb{R})}.$$
However, the above inequality implies that
\begin{multline*}\|J^\alpha _{t_0,t}f\|_{L^{p}(t_0,t_1;X)}\leq \|J^\alpha _{t_0,t}\|f\|_X\|_{L^{p}(t_0,t_1;\mathbb{R})}\leq\left[\dfrac{(t_1-t_0)^\alpha}{\Gamma(\alpha+1)}\right]\|\|f\|_X\|_{L^{p}(t_0,t_1;\mathbb{R})}
\\=\left[\dfrac{(t_1-t_0)^\alpha}{\Gamma(\alpha+1)}\right]\|f\|_{L^{p}(t_0,t_1;X)},\end{multline*}
i.e., $J_{t_0,t}^\alpha f\in L^{p}(t_0,t_1;X)$ and inequality \eqref{naturalinclusion} holds.
\end{proof}

For the completeness of the theory presented in this paper, in what follows we present another proof to Theorem \ref{minkowskiseq}, when $1\leq p<\infty$, which now is only based on several algebraic manipulations.

\begin{theorem}\label{continuity} If $1\leq p<\infty$ and $f\in L^p(t_0,t_1;X)$, then $J_{t_0,t}^\alpha f(t)$ belongs to $L^p(t_0,t_1;X)$. Furthermore, it holds that
\begin{equation*}\left[\int_{t_0}^{t_1}{\left\|J_{t_0,t}^\alpha f(t)\right\|^p_X}\,dt\right]^{1/p}\leq \left[\dfrac{(t_1-t_0)^\alpha}{\Gamma(\alpha+1)}\right] \|f\|_{L^p(t_0,t_1;X)}.\end{equation*}
\end{theorem}

\begin{proof} If we change the variables $s=t-w$, we obtain
\begin{multline*}\Gamma(\alpha)\left[\int_{t_0}^{t_1}{\left\|J_{t_0,t}^\alpha f(t)\right\|^p_X}\,dt\right]^{1/p}\leq\left[\int_{t_0}^{t_1}\left[\int_{t_0}^{t}(t-w)^{\alpha-1}\|f(w)\|_X\,dw\right]^p\,dt\right]^{1/p}\\
=\left[\int_{t_0}^{t_1}\left[\int_{0}^{t-t_0}s^{\alpha-1}\|f(t-s)\|_X\,ds\right]^p\,dt\right]^{1/p}.
\end{multline*}
Again, by changing the variables $t=r+t_0$, we deduce
$$\Gamma(\alpha)\left[\int_{t_0}^{t_1}{\left\|J_{t_0,t}^\alpha f(t)\right\|^p_X}\,dt\right]^{1/p}\leq\left[\int_{0}^{t_1-t_0}\left[\int_{0}^{r}s^{\alpha-1}\|f(r+t_0-s)\|_X\,ds\right]^p\,dr\right]^{1/p}.$$
Now, if we apply Corollary \ref{minkowskiminkowski} in the term on the right side of the above inequality, we get
$$\Gamma(\alpha)\left[\int_{t_0}^{t_1}{\left\|J_{t_0,t}^\alpha f(t)\right\|^p_X}\,ds\right]^{1/p}
\leq\int_{0}^{t_1-t_0}s^{\alpha-1}\left[\int_{s}^{t_1-t_0}\|f(r+t_0-s)\|_X^p\,dr\right]^{1/p}\,ds.$$

Finally, by choosing $r=l+s-t_0$, we achieve
\begin{multline*}\Gamma(\alpha)\left[\int_{t_0}^{t_1}{\left\|J_{t_0,t}^\alpha f(t)\right\|^p_X}\,ds\right]^{1/p}\leq \int_{0}^{t_1-t_0}s^{\alpha-1}\left[\int_{t_0}^{t_1-s}\|f(l)\|_X^p\,dl\right]^{1/p}\,ds\\
\leq\left[\int_{0}^{t_1-t_0}s^{\alpha-1}\,ds\right]\left[\int_{t_0}^{t_1}\|f(l)\|_X^p\,dl\right]^{1/p}=\left[\dfrac{(t_1-t_0)^\alpha}{\alpha}\right]\|f\|_{L^p(t_0,t_1;X)},\end{multline*}
what is equivalent to
$$\left[\int_{t_0}^{t_1}{\left\|J_{t_0,t}^\alpha f(t)\right\|^p_X}\,ds\right]^{1/p}\leq \left[\dfrac{(t_1-t_0)^\alpha}{\Gamma(\alpha+1)}\right] \|f\|_{L^p(t_0,t_1;X)},$$
as we wanted.
\end{proof}

As proved above, the fractional integral $J_{t_0,t}^\alpha$ defines a bounded linear operator from the Banach space $L^p(t_0,t_1;X)$ into itself, no matter what value $1\leq p\leq\infty$. Thus, following this line of thought, it is very natural to ask whether or not it is also a compact operator.

Let us begin by recalling J. Simon's result about the characterization of compact sets in $L^p(t_0,t_1;X)$, when $1\leq p<\infty$.

\begin{theorem}\label{simon}\cite[Theorem 1]{Si1} Let $1\leq p<\infty$. A set of functions $F\subset L^p(t_0,t_1;X)$ is relatively compact in $L^p(t_0,t_1;X)$ if, and only if:\vspace*{0.2cm}
\begin{itemize}
\item[(i)] $\left\{\int_{t_0^*}^{t_1^*}f(t)dt\,:\,f\in F\right\}$ is relatively compact in $X$, for every $t_0<t_0^*<t_1^*<t_1$;\vspace*{0.2cm}
\item[(ii)] $\lim_{h\rightarrow 0^+}\left[\sup_{f\in F}\left(\int_{t_0}^{t_1-h}\|f(t+h)-f(t)\|_X^p\,dt\right)^{1/p}\right]=0$.
\end{itemize}
\end{theorem}

\begin{remark} The case $p=\infty$ in the above theorem does not holds in general, as pointed out by J. Simon. In fact, if we define $\phi:[0,2]\rightarrow\mathbb{R}$ by
$$\phi(t)=\left\{\begin{array}{ll}0,&\textrm{for }\,t\in[0,1),\\1,&\textrm{for }\,t\in[1,2],\end{array}\right.$$
then $F:=\{\phi\}$ is compact in $L^\infty(0,2;\mathbb{R})$ but does not satisfies item $(ii)$, since
$$\esssup_{s\in [0,2-h]}\|\phi(s+h)-\phi(s)\|_{X}=1.$$
\end{remark}

Bearing Theorem \ref{simon} in mind we prove the following compactness result to the RL fractional integral.

\begin{theorem}\label{compactrieman} Let $1\leq p<\infty$ and $\alpha>0$. The bounded operator
$$J_{t_0,t}^\alpha:L^p(t_0,t_1;X)\rightarrow L^p(t_0,t_1;X)$$
is compact if, and only if, for any bounded set $F\subset L^p(t_0,t_1;X)$ it holds that
\begin{equation}\label{novahip}\left\{\begin{array}{l}\left\{J_{t_0,t_1^*}^{1+\alpha} f(t_1^*)-J_{t_0,t_0^*}^{1+\alpha} f(t_0^*)\,:\,f\in F\right\}\textrm{ is relatively compact in } X,\vspace*{0.2cm}\\
\textrm{for every }t_0<t_0^*<t_1^*<t_1.\end{array}\right.\end{equation}
\end{theorem}

\begin{proof} Given $t_0<t_0^*<t_1^*<t_1$, consider the linear operator %
$$\begin{array}{lclc}J_{t_0^*,t_1^*}:&L^p(t_0,t_1;X)&\rightarrow& X,\\
&f&\mapsto& J_{t_0^*,t_1^*}f:=\int_{t_0^*}^{t_1^*}f(s)\,ds.\end{array}$$

Observe that $J_{t_0^*,t_1^*}$ is a linear operator and also satisfies
$$\left\|J_{t_0^*,t_1^*}f\right\|_X\leq(t_1^*-t_0^*)^{1-(1/p)}\|f\|_{L^p(t_0,t_1;X)},$$
i.e., $J_{t_0^*,t_1^*}$ is a bounded operator from $L^p(t_0,t_1;X)$ into $X$.

Now, since we are assuming that $J_{t_0,t}^\alpha: L^p(t_0,t_1;X) \rightarrow L^p(t_0,t_1;X)$ is a compact operator, we conclude that
$$\begin{array}{lclc}J_{t_0^*,t_1^*}\circ J_{t_0,t}^\alpha:&L^p(t_0,t_1;X)&\rightarrow& X,\\
&f&\mapsto& \int_{t_0^*}^{t_1^*}J_{t_0,s}^\alpha f(s)\,ds,\end{array}$$
is a compact operator. Furthermore, notice that
\begin{multline*}\int_{t_0^*}^{t_1^*}J_{t_0,s}^\alpha f(s)\,ds=\dfrac{1}{\Gamma(\alpha)}\int_{t_0^*}^{t_1^*}\left[\int_{t_0}^s(s-w)^{\alpha-1} f(w)\,dw\right]\,ds\\
\hspace*{2.5cm}=\dfrac{1}{\Gamma(\alpha)}\int_{t_0^*}^{t_1^*}\left[\int_{t_0}^{t_0^*}(s-w)^{\alpha-1} f(w)\,dw\right]\,ds\\+\dfrac{1}{\Gamma(\alpha)}\int_{t_0^*}^{t_1^*}\left[\int_{t_0^*}^s(s-w)^{\alpha-1} f(w)\,dw\right]\,ds,
\end{multline*}
and therefore Theorem \ref{FubiniFubini} and Corollary \ref{FubiniFubini2} ensure that
\begin{multline*}\int_{t_0^*}^{t_1^*}J_{t_0,s}^\alpha f(s)\,ds=\dfrac{1}{\Gamma(\alpha)}\int_{t_0}^{t_0^*}\left[\int_{t_0^*}^{t_1^*}(s-w)^{\alpha-1} \,ds\right]f(w)\,dw\\+\dfrac{1}{\Gamma(\alpha)}\int_{t_0^*}^{t_1^*}\left[\int_{w}^{t_1^*}(s-w)^{\alpha-1}\,ds\right] f(w)\,dw,
\end{multline*}
what is equivalent to
$$\int_{t_0^*}^{t_1^*}J_{t_0,s}^\alpha f(s)\,ds=\dfrac{1}{\Gamma(\alpha+1)}\int_{t_0}^{t_1^*}(t_1^*-w)^{\alpha}f(w)\,dw-\dfrac{1}{\Gamma(\alpha+1)}\int_{t_0}^{t_0^*}(t_0^*-w)^{\alpha}f(w)\,dw.$$

Hence, if $F$ is any bounded subset of $L^p(t_0,t_1;X)$, we know that
$$\Big[J_{t_0^*,t_1^*}\circ J_{t_0,t}^\alpha \Big](F)=\left\{J_{t_0,t_1^*}^{1+\alpha} f(t_1^*)-J_{t_0,t_0^*}^{1+\alpha} f(t_0^*)\,:\,f\in F\right\},$$
is a relatively compact set in $X$, as we wanted.

Conversely, let $F\subset L^p(t_0,t_1;X)$ be a bounded set. In order to prove that $J_{t_0,t}^\alpha F$ is relatively compact in $L^p(t_0,t_1;X)$, we need to verify that this set satisfies items $(i)$ and $(ii)$ of Theorem \ref{simon}. Since
$$\left\{\int_{t_1^*}^{t_0^*}{J_{t_0,s}^\alpha f(s)}\,ds\,:\,f\in F\right\}=\left\{J_{t_0,t_1^*}^{1+\alpha} f(t_1^*)-J_{t_0,t_0^*}^{1+\alpha} f(t_0^*)\,:\,f\in F\right\},$$
item $(i)$ becomes a direct consequence of \eqref{novahip}. Thus, to complete this proof we only need to prove that $J_{t_0,t}^\alpha F$ satisfies item $(ii)$ of Theorem \ref{simon}.

Thus, observe that for each $h>0$ (sufficiently small)
\begin{multline*}\Gamma(\alpha)\big\|J_{t_0,s+h}^\alpha f(s+h)-J_{t_0,s}^\alpha f(s)\big\|_X\,\leq\int_{t_0}^s\big|(s-w)^{\alpha-1}-(s+h-w)^{\alpha-1}\big|\|f(w)\|_X\,dw\\
+\int_{s}^{s+h}(s+h-w)^{\alpha-1}\|f(w)\|_X\,dw.\end{multline*}
By changing the variable $w=s-r$, we obtain
\begin{multline*}\Gamma(\alpha)\big\|J_{t_0,s+h}^\alpha f(s+h)-J_{t_0,s}^\alpha f(s)\big\|_X\leq\int_{0}^{s-t_0}\big|r^{\alpha-1}-(r+h)^{\alpha-1}\big|\|f(s-r)\|_X\,dr\\
+\int_{-h}^{0}(r+h)^{\alpha-1}\|f(s-r)\|_X\,dr.\end{multline*}
Therefore, Minkowski's inequality ensures that
\begin{multline*}\Gamma(\alpha)\left(\int_{t_0}^{t_1-h}\big\|J_{t_0,s+h}^\alpha f(s+h)-J_{t_0,s}^\alpha f(s)\big\|_X^p ds\right)^{1/p}\\\leq \left(\int_{t_0}^{t_1-h}\left[\int_{0}^{s-t_0}\big|r^{\alpha-1}-(r+h)^{\alpha-1}\big|\|f(s-r)\|_X\,dr\right]^p\,ds\right)^{1/p}\\
 +\left(\int_{t_0}^{t_1-h}\left[\int_{-h}^{0}(r+h)^{\alpha-1}\|f(s-r)\|_X\,dr\right]^p\,ds\right)^{1/p}=:\mathcal{I}_h+\mathcal{J}_h.\end{multline*}

Firstly, let us prove that $\sup_{f\in F}\mathcal{I}_h\rightarrow0$, when $h\rightarrow0^+$. To this end, change the variable $s=x+t_0$ to deduce the identity
$$\mathcal{I}_h=\left(\int_{0}^{t_1-t_0-h}\left[\int_{0}^{x}\big|r^{\alpha-1}-(r+h)^{\alpha-1}\big|\|f(x+t_0-r)\|_X\,dr\right]^p\,dx\right)^{1/p}.$$
Thus, Corollary \ref{minkowskiminkowski} ensures that
$$\mathcal{I}_h\leq\int_{0}^{t_1-t_0-h}\big|r^{\alpha-1}-(r+h)^{\alpha-1}\big|\left[\int_{r}^{t_1-t_0-h}\|f(x+t_0-r)\|^p_X\,dx\right]^{1/p}\,dr,$$
which, by changing the variable $x=y+r-t_0$, becomes
$$\mathcal{I}_h\leq\int_{0}^{t_1-t_0-h}\big|r^{\alpha-1}-(r+h)^{\alpha-1}\big|\left[\int_{t_0}^{t_1-h-r}\|f(y)\|^p_X\,dy\right]^{1/p}\,dr.$$
Finally, by changing $r=t_1-h-z$ we obtain
\begin{multline*}\mathcal{I}_h\leq\int_{t_0}^{t_1-h}\big|(t_1-h-z)^{\alpha-1}-(t_1-z)^{\alpha-1}\big|\left[\int_{t_0}^{z}\|f(y)\|^p_X\,dy\right]^{1/p}\,dz\\
\leq\|F\|_{L^p(t_0,t_1;X)}\int_{t_0}^{t_1-h}\big|(t_1-h-z)^{\alpha-1}-(t_1-z)^{\alpha-1}|\,dz,\end{multline*}
where $\|F\|_{L^p(t_0,t_1;X)}=\sup_{f\in F}\|f\|_{L^p(t_0,t_1;X)}<\infty$. The conclusion follows now directly from
$$\lim_{h\rightarrow0^+}\int_{t_0}^{t_1-h}\big|(t_1-h-z)^{\alpha-1}-(t_1-z)^{\alpha-1}|\,dz=0.$$

Secondly, let us prove that $\sup_{f\in F}\mathcal{J}_h\rightarrow0$, when $h\rightarrow0^+$. By applying Theorem \ref{minkowski} we deduce that
$$\mathcal{J}_h\leq\int_{-h}^{0}(r+h)^{\alpha-1}\left[\int_{t_0}^{t_1-h}\|f(s-r)\|^p_X\,ds\right]^{1/p}\,dr$$
which, by changing of variables $s=x+r$ and then $r=-y$, is equivalent to
$$\mathcal{J}_h\leq\int_{0}^{h}(h-y)^{\alpha-1}\left[\int_{t_0+y}^{t_1+y-h}\|f(x)\|^p_X\,dx\right]^{1/p}\,dy.$$

Hence, we achieve the inequality
$$\mathcal{J}_h\leq\|F\|_{L^p(t_0,t_1;X)}\int_{0}^{h}(h-y)^{\alpha-1}\,dy,$$
where $\|F\|_{L^p(t_0,t_1;X)}=\sup_{f\in F}\|f\|_{L^p(t_0,t_1;X)}<\infty$. The conclusion follows now directly from
$$\lim_{h\rightarrow0^+}\int_{0}^{h}(h-y)^{\alpha-1}\,dy=0.$$
\end{proof}

\begin{corollary}\label{coroaux} Let $1\leq p<\infty$ and $\alpha>0$. If $X$ is a finite dimensional Banach space, then the bounded operator
$$J_{t_0,t}^\alpha:L^p(t_0,t_1;X)\rightarrow L^p(t_0,t_1;X)$$
is compact.
\end{corollary}

\begin{proof} Just observe that for $f\in L^p(t_0,t_1;X)$ and $t^*\in(t_0,t_1)$, Holder's inequality ensures that
\begin{multline*}\left\|\int_{t_0}^{t^*}(t^*-s)^{\alpha}f(s)\,ds\right\|_X\leq
\left[\int_{t_0}^{t^*}(t^*-s)^{(\alpha p)/(p-1)}\,ds\right]^{(p-1)/p}\|f\|_{L^p(t_0,t_1;X)}\\=\left[\dfrac{p-1}{(\alpha+1)p-1}\right]^{(p-1)/p}(t^*-t_0)^{\alpha+1-(1/p)}\|f\|_{L^p(t_0,t_1;X)}.\end{multline*}

The above estimate allows us to conclude that for any $F\subset L^p(t_0,t_1;X)$ bounded, the set $\{J_{t_0,t_1^*}^{1+\alpha} f(t_1^*)-J_{t_0,t_0^*}^{1+\alpha} f(t_0^*)\,:\,f\in F\}$ is also bounded in $X$. Since the dimension of $X$ is finite, then we deduce that %
$$\{J_{t_0,t_1^*}^{1+\alpha} f(t_1^*)-J_{t_0,t_0^*}^{1+\alpha} f(t_0^*)\,:\,f\in F\}$$
is relatively compact in $X$. Thus, Theorem \ref{compactrieman} gives us the compactness of the RL fractional integral operator in $L^p(t_0,t_1;X)$.
\end{proof}

To deal with the compactness of RL fractional integral in $L^\infty(t_0,t_1;X)$, let us first prove the following result.

\begin{theorem}\label{cntlinf} If $\alpha>0$ and $f\in L^\infty(t_0,t_1;X)$, it holds that $J_{t_0,t}^\alpha f\in C^0([t_0,t_1];X)$.
\end{theorem}

\begin{proof} Initially, Holder's inequality ensures the estimate
$$\|J_{t_0,t}^\alpha f(t)\|_X\leq\dfrac{1}{\Gamma(\alpha)}\int_{t_0}^t{(t-w)^{\alpha-1}\|f(w)\|_X}\,dw\leq\dfrac{(t_1-t_0)^\alpha}{\Gamma(\alpha+1)}\|f\|_{L^\infty(t_0,t_1;X)},$$
for any $t\in[t_0,t_1]$, i.e, $J_{t_0,t}^\alpha f(t)$ exists for every $t\in[t_0,t_1]$.

Now, for $t,s\in[t_0,t_1]$ (assume without loss of generality that $t>s$) we have
\begin{multline*}\|J_{t_0,t}^\alpha f(t)-J_{t_0,s}^\alpha f(s)\|_X\leq \dfrac{1}{\Gamma(\alpha)}\int_{s}^{t}{(t-w)^{\alpha-1}\|f(w)\|_X}\,dw\\
+\dfrac{1}{\Gamma(\alpha)}\int_{t_0}^{s}{\big|(s-w)^{\alpha-1}-(t-w)^{\alpha-1}\big|\|f(w)\|_X}\,dw.\end{multline*}
But then, by Holder's inequality, we have that
\begin{multline*}\|J_{t_0,t}^\alpha f(t)-J_{t_0,s}^\alpha f(s)\|_X\leq \left[\dfrac{(t-s)^\alpha}{\Gamma(\alpha+1)}\right.\\\left.+\left|\dfrac{(t-s)^\alpha}{\Gamma(\alpha+1)}+\dfrac{(s-t_0)^\alpha}{\Gamma(\alpha+1)}-\dfrac{(t-t_0)^\alpha}{\Gamma(\alpha+1)}\right|\right]\|f\|_{L^\infty(t_0,t_1;X)}.\end{multline*}
The above estimate ensures that $J_{t_0,t}^\alpha f(t)$ is continuous in $[t_0,t_1]$.
\end{proof}

Last result allows us to understand that the compactness of RL fractional integral as an operator from $L^\infty(t_0,t_1;X)$ into itself can be proved by considering a characterization of compacts sets in $C^0([t_0,t_1];X)$, which has an interesting formulation proved by J. Simon. It worths to emphasize that J. Simon's Theorem fits better our needs then Arzela-Ascoli characterization of compact sets in $C^0([t_0,t_1];X)$.

\begin{theorem}\label{Arzelasimon}\cite[Theorem 1]{Si1} A set of functions $F\subset C^0([t_0,t_1];X)$ is relatively compact if, and only if:\vspace*{0.2cm}
\begin{itemize}
\item[(i)] $\left\{\int_{t_0^*}^{t_1^*}f(t)dt\,:\,f\in F\right\}$ is relatively compact in $X$, for every $t_0<t_0^*<t_1^*<t_1$;\vspace*{0.2cm}
\item[(ii)] $\lim_{h\rightarrow 0^+}\left[\sup_{f\in F}\left(\sup_{t\in[t_0,t_1-h]}\|f(t+h)-f(t)\|_X\right)\right]=0$.
\end{itemize}
\end{theorem}

With these acquired insights in mind, we present the following compactness criteria to the RL fractional integral in $L^\infty(t_0,t_1;X)$.

\begin{theorem} Let $\alpha>0$. The bounded operator
$$J_{t_0,t}^\alpha:L^\infty(t_0,t_1;X)\rightarrow L^\infty(t_0,t_1;X)$$
is compact if, and only if, for any bounded set $F\subset L^\infty(t_0,t_1;X)$ it holds that
\begin{equation}\label{novahip1}\left\{\begin{array}{l}\left\{J_{t_0,t_1^*}^{1+\alpha} f(t_1^*)-J_{t_0,t_0^*}^{1+\alpha} f(t_0^*)\,:\,f\in F\right\}\textrm{ is relatively compact in } X,\vspace*{0.2cm}\\
\textrm{for every }t_0<t_0^*<t_1^*<t_1.\end{array}\right.\end{equation}
\end{theorem}

\begin{proof} If we assume that $J_{t_0,t}^\alpha$ is a compact operator from $L^\infty(t_0,t_1;X)$ into itself, then proceeding exactly as in the proof of Theorem \ref{compactrieman} we deduce that
$$\left\{J_{t_0,t_1^*}^{1+\alpha} f(t_1^*)-J_{t_0,t_0^*}^{1+\alpha} f(t_0^*)\,:\,f\in F\right\}$$
is relatively compact in $X$, for any bounded set $F\subset L^\infty(t_0,t_1;X)$.

Conversely, let $F\subset L^\infty(t_0,t_1;X)$ be a bounded set. In order to prove that $J_{t_0,t}^\alpha F$ is relatively compact in $L^\infty(t_0,t_1;X)$ (what in this case is equivalent to prove that $J_{t_0,t}^\alpha F$ is relatively compact in $C^0([t_0,t_1];X)$, see Theorem \ref{cntlinf}), we need to verify that this set satisfies items $(i)$ and $(ii)$ of Theorem \ref{Arzelasimon}. Since
$$\left\{\int_{t_1^*}^{t_0^*}{J_{t_0,s}^\alpha f(s)}\,ds\,:\,f\in F\right\}=\left\{J_{t_0,t_1^*}^{1+\alpha} f(t_1^*)-J_{t_0,t_0^*}^{1+\alpha} f(t_0^*)\,:\,f\in F\right\},$$
item $(i)$ becomes a direct consequence of \eqref{novahip1}. Thus, to complete this proof we need to prove that $J_{t_0,t}^\alpha F$ satisfies item $(ii)$ of Theorem \ref{Arzelasimon}.

Observe that for each $h>0$ (sufficiently small)
\begin{multline*}\Gamma(\alpha)\big\|J_{t_0,s+h}^\alpha f(s+h)-J_{t_0,s}^\alpha f(s)\big\|_X\,\leq\int_{t_0}^s\Big|(s-w)^{\alpha-1}-(s+h-w)^{\alpha-1}\Big|\|f(w)\|_X\,dw\\
+\int_{s}^{s+h}(s+h-w)^{\alpha-1}\|f(w)\|_X\,dw.\end{multline*}
But then, Holder's inequality ensures that
\begin{multline*}\Gamma(\alpha+1)\big\|J_{t_0,s+h}^\alpha f(s+h)-J_{t_0,s}^\alpha f(s)\big\|_X\,\\\leq\Big[h^\alpha+\Big|h^\alpha+(s-t_0)^\alpha-(s+h-t_0)^\alpha\Big|\Big]\|f\|_{L^\infty(t_0,t_1;X)},\end{multline*}
for every $s\in[t_0,t_1-h]$. Since $[t_0,t_1-h]\ni s\mapsto h^\alpha+\big|h^\alpha+(s-t_0)^\alpha-(s+h-t_0)^\alpha\big|$ is a non increasing function, we deduce that
$$\sup_{f\in F}\left(\sup_{s\in[t_0,t_1-h]}\|J_{t_0,s+h}^\alpha f(s+h)-J_{t_0,s}^\alpha f(s)\|_X\right)\leq\left[\dfrac{h^\alpha}{\Gamma(\alpha+1)}\right]\|F\|_{L^\infty(t_0,t_1;X)},$$
where $\|F\|_{L^\infty(t_0,t_1;X)}=\sup_{f\in F}\|f\|_{L^\infty(t_0,t_1;X)}<\infty$. The conclusion follows now directly.
\end{proof}

\begin{corollary} Let $\alpha>0$. If $X$ is a finite dimensional Banach space, then the bounded operator
$$J_{t_0,t}^\alpha:L^\infty(t_0,t_1;X)\rightarrow L^\infty(t_0,t_1;X)$$
is compact.
\end{corollary}

\begin{proof} Just follow the same steps of the proof of Corollary \ref{coroaux}.
\end{proof}

\subsection{On the Continuity in the Order of Integration} Recall that item $(ii)$ of Remark \ref{remark5} does not discuss the general case $f\in L^p(t_0,t_1;X)$. This is why we present our next results. It worths to emphasize that besides being a generalization of the classical result discussed by S. G. Samko et al. in his book \cite{SaKiMa1}, the proof we present to Lemma \ref{auxtextfunc} discuss a technique that motivates the subject discussed in Theorem \ref{gerinf}.

\begin{lemma}\label{auxtextfunc}
  Let $1\leq p\leq\infty$. Then
\begin{equation*}\lim_{\alpha\rightarrow0^+}{\big\|J_{t_0,t}^\alpha\phi-\phi\big\|_{L^p(t_0,t_1;{X})}}=0,\vspace*{0.2cm}\end{equation*}
for each $\phi\in C_c^\infty([t_0,t_1];X)$.
\end{lemma}

\begin{proof} We may apply integration by parts to deduce that
$$J_{t_0,t}^\alpha\phi(t)=\dfrac{1}{\Gamma(\alpha+1)}\int_{t_0}^t(t-s)^\alpha\phi^\prime(s)\,ds,$$
for every $t\in [t_0,t_1]$.

Consider $\{r_n\}_{n=1}^\infty$ a sequence of positive real values such that $\lim_{n\rightarrow\infty}{r_n}=0$. We may assume, without loss of generality, that $r_n\in(0,1/2)$, for every $n\in\mathbb{N}$.

Thus, we have that
$$J_{t_0,t}^{r_n} \phi(t)-\phi(t)=\int_{t_0}^t{\left[\dfrac{(t-s)^{r_n}}{\Gamma(r_n+1)}-1\right]\phi^\prime(s)}\,ds,$$
for every $t\in[t_0,t_1]$.

Now, for each $t_0\leq s< t\leq t_1$, consider the continuously differentiable function $\nu_{t,s}:[0,1]\rightarrow\mathbb{R}$ given by
$$\nu_{t,s}(w)=\dfrac{(t-s)^w}{\Gamma(w+1)}.$$
Since
\begin{equation}\label{ajudaequili}\dfrac{d}{dw}\Big[\nu_{t,s}(w)\Big]=\dfrac{(t-s)^{w}\big[\ln(t-s)-\psi(w+1)\big]}{\Gamma(w+1)},\end{equation}
for every $w\in[0,1]$, the Main Value Theorem ensures that
$$\left|\dfrac{(t-s)^{r_n}}{\Gamma(r_n+1)}-1\right|
=r_n\left[\dfrac{(t-s)^{\xi_{t,s,n}}\big|\ln(t-s)-\psi(\xi_{t,s,n}+1)\big|}{\Gamma(\xi_{t,s,n}+1)}\right],$$
for some $\xi_{t,s,n}\in(0,r_n).$ Above, function $\psi(z)$ stands for the Digamma function (for more details see \cite{AbSt1}).

In this way, we obtain that
\begin{multline}\label{ultimadesinova-01}\left\|J_{t_0,t}^{r_n} \phi(t)-\phi(t)\right\|_X\leq
r_n\\\times\int_{t_0}^t{\left[\dfrac{(t-s)^{\xi_{t,s,n}-(1/2)}\Big[\big|(t-s)^{1/2}\ln(t-s)\big|+(t-s)^{1/2}\big|\psi(\xi_{t,s,n}+1)\big|\Big]}{\Gamma(\xi_{t,s,n}+1)}\right]\|\phi^\prime(s)\|_X}\,ds,\end{multline}
for every $t\in[t_0,t_1]$.  Since function $r^{1/2}\ln(r)$ is bounded in $[0,t_1-t_0]$, we deduce that
\begin{equation}\label{ultimadesinova}\left\|J_{t_0,t}^{r_n} \phi(t)-\phi(t)\right\|_X\leq
\Big[M_1r_n\sigma_n(t)\Big]\max_{s\in[t_0,t_1]}\|\phi^\prime(s)\|_X,\end{equation}
for almost every $t\in[t_0,t_1]$, where
\begin{multline}\label{ultimadesinova+01} \sigma_n(t)=\int_{t_0}^t{(t-s)^{\xi_{t,s,n}-(1/2)}}\,ds\qquad\textrm{and}\\
M_1=\max_{r\in[0,1]}{\left[\dfrac{1+\big|\psi(r+1)\big|}{\Gamma(r+1)}\right]}\left[\max_{r\in[0,t_1-t_0]}{|r^{1/2}\ln(r)|}+(t_1-t_0)^{1/2}\right].\end{multline}

Now just observe that
$$(t-s)^{\xi_{t,s,n}-(1/2)}\leq(t-s)^{-(1/2)}+(t-s)^{r_n-(1/2)},$$
for every $t_0\leq s< t\leq t_1$. Thus, we deduce the existence of $M_2>0$ such that
\begin{equation}\label{ultimadesinova2}\sigma_n(t)\leq {\dfrac{(t_1-t_0)^{1/2}}{1/2}+\dfrac{(t_1-t_0)^{r_n+(1/2)}}{r_n+(1/2)}}\leq M_2,\end{equation}
for every $t\in[t_0,t_1]$ and $n\in\mathbb{N}$.

Taking into account \eqref{ultimadesinova} and \eqref{ultimadesinova2}, we obtain
\begin{equation*}\left\|J_{t_0,t}^{r_n} \phi-\phi\right\|_{L^p(t_0,t_1;X)}\leq
\Big[M_1M_2(t_1-t_0)^{1/p}\max_{s\in[t_0,t_1]}\|\phi^\prime(s)\|_X\Big]r_n,\end{equation*}
for $1\leq p<\infty$, and
\begin{equation*}\left\|J_{t_0,t}^{r_n} \phi-\phi\right\|_{L^\infty(t_0,t_1;X)}\leq
\Big[M_1M_2\max_{s\in[t_0,t_1]}\|\phi^\prime(s)\|_X\Big]r_n.\end{equation*}

Therefore, $\{J_{t_0,t}^{r_n}\phi\}_{n=1}^\infty$ converges to $\phi$ in the topology of $L^p(t_0,t_1;X)$. If we recall that sequence $\{r_n\}_{n=1}^\infty$ was chosen arbitrarily, we achieve the desired result.  \end{proof}

\begin{theorem}\label{auxtextfunc2}
  Let $1\leq p\leq\infty$. It holds that
\begin{equation*}\lim_{\alpha\rightarrow0^+}{\big\|J_{t_0,t}^\alpha f-f\big\|_{L^p(t_0,t_1;{X})}}=0,\end{equation*}
for every $f\in L^p(t_0,t_1;{X})$.
\end{theorem}

\begin{proof} Let $f\in L^p(t_0,t_1;{X})$. For any $\varepsilon>0$,  choose $\phi_\varepsilon\in C_c^\infty([t_0,t_1];X)$ such that
\begin{equation}\label{desinovasera}\|f-\phi_\varepsilon\|_{L^p(t_0,t_1;X)}\leq \min{\left\{\dfrac{\varepsilon\Gamma(\alpha+1)}{2(t_1-t_0)^\alpha},\dfrac{\varepsilon}{2}\right\}}.\end{equation}

Now observe that
\begin{multline*}\big\|J_{t_0,t}^\alpha f-f\big\|_{L^p(t_0,t_1;{X})}\leq\big\|J_{t_0,t}^\alpha f-J_{t_0,t}^\alpha \phi_\varepsilon\big\|_{L^p(t_0,t_1;{X})}+\big\|J_{t_0,t}^\alpha \phi_\varepsilon-\phi_\varepsilon\big\|_{L^p(t_0,t_1;{X})}\\+\big\|\phi_\varepsilon-f\big\|_{L^p(t_0,t_1;{X})}.\end{multline*}

Thus, Theorem \ref{minkowskiseq}, Lemma \ref{auxtextfunc} and \eqref{desinovasera} ensure that
$$\lim_{\alpha\rightarrow0^+}\big\|J_{t_0,t}^\alpha f-f\big\|_{L^p(t_0,t_1;{X})}\leq\varepsilon.$$
Since $\varepsilon>0$ was chosen arbitrarily, the proof of the theorem is complete.\end{proof}

We finally present the result that discuss the uniform continuity of the RL fractional integral. It worths to emphasize that this proof is based on the proof given in \cite[Theorem 2.6]{SaKiMa1}.

\begin{theorem}\label{uniformconti01}
Let $1\leq p\leq\infty$. For any $\alpha_0>0$ it holds that
\begin{equation*}\lim_{\alpha\rightarrow\alpha_0}{\big\|J_{t_0,t}^\alpha -J_{t_0,t}^{\alpha_0}\big\|_{\mathcal{L}(L^p(t_0,t_1;{X}),L^p(t_0,t_1;{X}))}}=0.\end{equation*}
\end{theorem}

\begin{proof} Let $\alpha_0>0$ and observe that for $\alpha>0$ we have
\begin{multline*}\|J_{t_0,t}^\alpha f(t)-J_{t_0,t}^{\alpha_0} f(t)\|_X\leq\dfrac{1}{\Gamma(\alpha)}\int_{t_0}^{t}\left|(t-s)^{\alpha-1}-(t-s)^{\alpha_0-1}\right|\|f(s)\|_X\,ds
\\+\left|\dfrac{\Gamma(\alpha_0)-\Gamma(\alpha)}{\Gamma(\alpha)\Gamma(\alpha_0)}\right|\int_{t_0}^t(t-s)^{\alpha_0-1}\|f(s)\|_X\,ds,\end{multline*}
for almost every $t\in[t_0,t_1]$. By changing the variable $s=t-w$ in the first term on the right side of the above inequality, we obtain
\begin{multline*}\|J_{t_0,t}^\alpha f(t)-J_{t_0,t}^{\alpha_0} f(t)\|_X\leq\dfrac{1}{\Gamma(\alpha)}\int_{0}^{t-t_0}\left|w^{\alpha-1}-w^{\alpha_0-1}\right|\|f(t-w)\|_X\,dw
\\+\left|\dfrac{\Gamma(\alpha_0)-\Gamma(\alpha)}{\Gamma(\alpha)\Gamma(\alpha_0)}\right|\int_{t_0}^t(t-s)^{\alpha_0-1}\|f(s)\|_X\,ds,\end{multline*}
for almost every $t\in[t_0,t_1]$. But then, Minkowski's inequality ensures that
\begin{multline*}\|J_{t_0,t}^\alpha f(t)-J_{t_0,t}^{\alpha_0} f(t)\|_{L^p(t_0,t_1;X)}\\\leq\dfrac{1}{\Gamma(\alpha)}\left\{\int_{t_0}^{t_1}\left[\int_{0}^{t-t_0}\left|w^{\alpha-1}-w^{\alpha_0-1}\right|\|f(t-w)\|_X\,dw\right]^p\,dt\right\}^{1/p}
\\+\left|\dfrac{\Gamma(\alpha_0)-\Gamma(\alpha)}{\Gamma(\alpha)\Gamma(\alpha_0)}\right|\left\{\int_{t_0}^{t_1}\left[\int_{t_0}^t(t-s)^{\alpha_0-1}\|f(s)\|_X\,ds\right]^p\,dt\right\}^{1/p}.\end{multline*}

Again, changing the variable $t=r+t_0$ in the first term on the right side of the above inequality, we get
\begin{multline*}\|J_{t_0,t}^\alpha f(t)-J_{t_0,t}^{\alpha_0} f(t)\|_{L^p(t_0,t_1;X)}\\\leq\dfrac{1}{\Gamma(\alpha)}\left\{\int_{0}^{t_1-t_0}\left[\int_{0}^{r}\left|w^{\alpha-1}-w^{\alpha_0-1}\right|\|f(r+t_0-w)\|_X\,dw\right]^p\,dr\right\}^{1/p}
\\+\left|\dfrac{\Gamma(\alpha_0)-\Gamma(\alpha)}{\Gamma(\alpha)\Gamma(\alpha_0)}\right|\left\{\int_{t_0}^{t_1}\left[\int_{t_0}^t(t-s)^{\alpha_0-1}\|f(s)\|_X\,ds\right]^p\,dt\right\}^{1/p}.\end{multline*}

If we apply Corollary \ref{minkowskiminkowski} and Theorem \ref{continuity} respectively in the first and second terms on the right side of the above inequality, we obtain
\begin{multline*}\|J_{t_0,t}^\alpha f(t)-J_{t_0,t}^{\alpha_0} f(t)\|_{L^p(t_0,t_1;X)}\\\leq\dfrac{1}{\Gamma(\alpha)}\int_{0}^{t_1-t_0}\left|w^{\alpha-1}-w^{\alpha_0-1}\right|\left[\int_{w}^{t_1-t_0}\|f(r+t_0-w)\|^p_X\,dr\right]^{1/p}\,dw
\\+\left[\dfrac{(t_1-t_0)^{\alpha_0}\big|\Gamma(\alpha_0)-\Gamma(\alpha)\big|}{\Gamma(\alpha)\Gamma(\alpha_0+1)}\right]\|f\|_{L^p(t_0,t_1;X)}.\end{multline*}

Observe that by changing the variable $r=h-t_0+w$ in the inner integral of the first term in the right side of the above inequality, we obtain
\begin{multline*}\|J_{t_0,t}^\alpha f(t)-J_{t_0,t}^{\alpha_0} f(t)\|_{L^p(t_0,t_1;X)} \leq\int_{0}^{t_1-t_0}\dfrac{\left|w^{\alpha-1}-w^{\alpha_0-1}\right|}{\Gamma(\alpha)}\left[\int_{t_0}^{t_1-w}\|f(h)\|^p_X\,dh\right]^{1/p}\,dw
\\+\left[\dfrac{(t_1-t_0)^{\alpha_0}\big|\Gamma(\alpha_0)-\Gamma(\alpha)\big|}{\Gamma(\alpha)\Gamma(\alpha_0+1)}\right]\|f\|_{L^p(t_0,t_1;X)},
\end{multline*}
and therefore
\begin{multline*}\|J_{t_0,t}^\alpha f(t)-J_{t_0,t}^{\alpha_0} f(t)\|_{L^p(t_0,t_1;X)} \\\leq \underbrace{\left[\dfrac{1}{\Gamma(\alpha)}\int_{0}^{t_1-t_0}\left|w^{\alpha-1}-w^{\alpha_0-1}\right|\,dw
+\dfrac{(t_1-t_0)^{\alpha_0}\big|\Gamma(\alpha_0)-\Gamma(\alpha)\big|}{\Gamma(\alpha)\Gamma(\alpha_0+1)}\right]}_{=\eta_\alpha}\|f\|_{L^p(t_0,t_1;X)}.
\end{multline*}

Finally, since $\eta_\alpha\rightarrow 0$, when $\alpha\rightarrow\alpha_0$, because of Theorem \ref{gedominatedconv}  and the continuity of Gamma function in the positive real line, we deduce that
\begin{multline*}\lim_{\alpha\rightarrow\alpha_0}{\big\|J_{t_0,t}^\alpha -J_{t_0,t}^{\alpha_0}\big\|_{\mathcal{L}(L^p(t_0,t_1;{X}),L^p(t_0,t_1;{X}))}}\\=\lim_{\alpha\rightarrow\alpha_0}\left[{\sup_{f\in L^p(t_0,t_1;X)\setminus\{0\}}\dfrac{\big\|J_{t_0,t}^\alpha f-J_{t_0,t}^{\alpha_0}f\big\|_{L^p(t_0,t_1;{X})}}{\|f\big\|_{L^p(t_0,t_1;{X})}}}\right]\leq\lim_{\alpha\rightarrow\alpha_0}{\eta_\alpha}=0.\end{multline*}
\end{proof}

\subsection{Semigroups and RL fractional integral}

In this last subsection we discuss the uniform continuity of RL fractional integral when $\alpha\rightarrow0^+$, in order to complement Theorem \ref{uniformconti01}. However, to do this discussion we need to recall the notions about semigroups of linear operators. Two classical sources on this subject are \cite{HiPh1,Paz1}.

\begin{definition} A family $\{T(t):t\geq0\}\subset\mathcal{L}(X)$ is called a semigroup of bounded linear operators (or just semigroup) in $X$ if:
\begin{itemize}
\item[(i)] $T(0)=Id$, where $Id:X\rightarrow X$ denotes de identity operator in $X$;
\item[(ii)] $T(t+s)=T(t)T(s)$, for every $t,s\geq0$.
\end{itemize}
\end{definition}

\begin{definition} If $\{T(t):t\geq0\}\subset\mathcal{L}(X)$ is a semigroup in $X$, we say that:
\begin{itemize}
\item[(i)] The semigroup is strongly continuous (or just $C_0-$semigroup) if
$$\lim_{t\rightarrow0^+}\|T(t)x-x\|_X=0,$$
for every $x\in X$;
\item[(ii)] The semigroup is uniformly continuous if
$$\lim_{t\rightarrow0^+}\|T(t)-Id\|_{\mathcal{L}(X)}=0.$$
\item[(iii)] The infinitesimal generator of the semigroup is a linear operator $A:D(A)\subset X\rightarrow X$, such that
$$D(A):=\left\{x\in X:\lim_{t\rightarrow0^+}\dfrac{\|T(t)x-x\|_X}{t}\textrm{ exists}\right\},$$
and
$$Ax=\lim_{t\rightarrow0^+}\dfrac{\|T(t)x-x\|_X}{t},$$
for any $x\in D(A)$.
\end{itemize}
\end{definition}

In order to prove the main theorem of this subsection, we shall recall the following classical result (see \cite[Corollary 1.4 and Theorem 2.4]{Paz1} for details on the proof):

\begin{theorem}\label{mainsemigrupo} Let $\{T(t):t\geq0\}$ be a $C_0-$semigroup in $X$ and $A:D(A)\subset X\rightarrow X$ its infinitesimal generator.
\begin{itemize}
\item[(i)] For any $x\in X$ and $t\geq0$ we have that
$$\lim_{h\rightarrow0^+}\left[\dfrac{1}{h}\int_{t}^{t+h}T(s)x\,ds\right]=T(t)x,$$
in the topology of $X$.\vspace*{0.2cm}
\item[(ii)] For $x\in X$, it holds that
$$\int_0^tT(s)x\,ds\in D(A)\qquad\textrm{and}\qquad A\left(\int_0^tT(s)x\,ds\right)=T(t)x-x.\vspace*{0.2cm}$$
\item[(iii)] For $x\in D(A)$ and $t\geq0$, we have that $T(t)x\in D(A)$ and
$$\dfrac{d}{dt}T(t)x=AT(t)x=T(t)Ax.\vspace*{0.2cm}$$
\item[(iv)] $\{T(t):t\geq0\}$ is uniformly continuous if, and only if, $A\in\mathcal{L}(X)$.
\end{itemize}
\end{theorem}

It is widely known, as pointed by item (i) of Remark \ref{remark5}, that RL fractional integral can be reinterpreted as the convolution $g_\alpha*f_{t_0}(t)$, for almost every $t\in[t_0,t_1]$. Thus, the semigroup property of the family of functions $\{g_\alpha(t):\alpha>0\}$ and convolution associativity allow us to conclude that $\{J_{t_0,t}^{\alpha}:\alpha>0\}$ inherits this property. In order to keep this work self contained, bellow we state this result and make a reference to its detailed proof.

\begin{theorem}\label{covproper} Let $1\leq p\leq\infty$. Then the family $\{J_{t_0,t}^\alpha:\alpha\geq0\}$ defines a $C_0-$semigroup in $L^p(t_0,t_1;X)$.
\end{theorem}

\begin{proof} See \cite[Propositions 2.24 and 2.35 for details on the proof]{Car1}.
\end{proof}

Now we address the infinitesimal generator of the $C_0-$semigroup $\{J_{t_0,t}^\alpha:\alpha\geq0\}$. It worths to point out that our next result was inspired in \cite[Theorem 23.16.1]{HiPh1} and in the subjects we have discussed in the proof of Lemma \ref{auxtextfunc}.

\begin{theorem}\label{gerinf} Let $1\leq p\leq \infty$ and assume that $A:D(A)\subset L^p(t_0,t_1;X)\rightarrow L^p(t_0,t_1;X)$ is the infinitesimal generator of the $C_0-$semigroup $\{J_{t_0,t}^\alpha:\alpha\geq0\}$ in $L^p(t_0,t_1;X)$. Then $f\in D(A)$ if, and only if,
$$\int_{t_0}^t\ln(t-s)f(s)\,ds$$
is absolutely continuous from $[t_0,t_1]$ into $X$ and its derivative belongs to $L^p(t_0,t_1;X)$. Moreover, we have
\begin{equation}\label{charac}Af(t)=-\psi(1)f(t)+\dfrac{d}{dt}\left[\int_{t_0}^t\ln(t-s)f(s)\,ds\right],\end{equation}
for almost every $t\in[t_0,t_1]$, where $\psi(t)$ denotes the Digamma function (note that $-\psi(1)$ is the Euler–Mascheroni constant).
\end{theorem}

\begin{proof} For any $f\in D(A)$ item $(iii)$ of Theorem \ref{mainsemigrupo} ensures that
\begin{equation}\label{derlpha}\dfrac{d}{d\alpha}\Big[J_{t_0,t}^\alpha f\Big]=J_{t_0,t}^\alpha Af,\end{equation}
for every $\alpha\geq0$.

Thus, identity \eqref{derlpha} ensures that
$$\dfrac{d}{d\alpha}\left[\dfrac{1}{\Gamma(\alpha+1)}\int_{t_0}^t(t-s)^{\alpha}f(s)\,ds\right]=\dfrac{d}{d\alpha}\Big[J_{t_0,t}^{\alpha+1}f(t)\Big]=J_{t_0,t}^{\alpha+1}Af(t),$$
for almost every $t\in[t_0,t_1]$ and every $\alpha\geq0$, which, by Theorem \ref{leibnizint} and equation \eqref{ajudaequili}, is equivalent to
\begin{equation}\label{ultimanovidade01211}\int_{t_0}^t\underbrace{\left\{\dfrac{(t-s)^{\alpha}\big[\ln(t-s)-\psi(\alpha+1)\big]}{\Gamma(\alpha+1)}\right\}}_{=:\rho_\alpha(t,s)}f(s)\,ds=J_{t_0,t}^{\alpha+1}Af(t),\end{equation}
for almost every $t\in[t_0,t_1]$ and every $\alpha\geq0$. Above the symbol $\psi(z)$ denotes the Digamma function.

At this point, observe that a consequence of Theorems \ref{auxtextfunc2} and \ref{covproper} is that there exists a sequence $\{\tau_n\}_{n=1}^\infty$ of positive real numbers such that $\tau_n\rightarrow0$, when $n\rightarrow\infty$,
\begin{equation}\label{ultimanovidade01}\lim_{n\rightarrow\infty}J^{\tau_n+1}_{t_0,t}Af(t)=J^{1}_{t_0,t}Af(t),\end{equation}
for almost every $t\in[t_0,t_1]$ and $0<\tau_n<1/2$, for every $n\in\mathbb{N}$.

Then, for almost every $t\in[t_0,t_1]$ it holds that:
\begin{itemize}
\item[(i)] If we consider $[t_0,t]\ni s\mapsto\big[\ln(t-s)-\psi(1)\big]f(s)$, the continuity of the Gamma and Digamma functions in the positive real line ensure that the Bochner measurable functions $[t_0,t]\ni s\mapsto\rho_{\tau_n}(t,s)f(s)$ satisfies
$$\lim_{n\rightarrow\infty}\left\|\rho_{\tau_n}(t,s)f(s)-\big[\ln(t-s)-\psi(1)\big]f(s)\right\|_X,$$
for almost every $s\in[t_0,t]$; \vspace*{0.2cm}

\item[(ii)] Functions $[t_0,t]\ni s\mapsto(t-s)^{\tau_n-(1/2)}\|f(s)\|_X$ and $[t_0,t]\ni s\mapsto(t-s)^{-(1/2)}\|f(s)\|_X$ are Lebesgue integrable and also satisfies
$$\lim_{n\rightarrow\infty}\left|(t-s)^{\tau_n-(1/2)}\|f(s)\|_X-(t-s)^{-(1/2)}\|f(s)\|_X\right|=0,$$
for almost every $s\in[t_0,t]$; \vspace*{0.2cm}
\item[(iii)] $\|\rho_{\tau_n}(s)f(s)\|_X\leq M(t-s)^{\tau_n-(1/2)}\|f(s)\|_X,$ for almost every $s\in[t_0,t]$, where
$$M=\max_{r\in[0,1]}{\left[\dfrac{1+\big|\psi(r+1)\big|}{\Gamma(r+1)}\right]}\left[\max_{r\in[0,t_1-t_0]}{|r^{1/2}\ln(r)|}+(t_1-t_0)^{1/2}\right],$$
(confront with \eqref{ultimadesinova-01}, \eqref{ultimadesinova} and \eqref{ultimadesinova+01});\vspace*{0.2cm}
\item[(iv)] Finally, we also have
\begin{multline*}\hspace*{0.4cm}\lim_{n\rightarrow\infty}\left[M\int_{t_0}^t(t-s)^{\tau_n-(1/2)}\|f(s)\|_X\,ds\right]
\\=\lim_{n\rightarrow\infty}\Big[M\Gamma(\tau_n+(1/2))J_{t_0,t}^{\tau_n+(1/2)}\|f(t)\|_X\Big]
=M\Gamma(1/2)J_{t_0,t}^{1/2}\|f(t)\|_X\\=\int_{t_0}^tM(t-s)^{-(1/2)}\|f(s)\|_X\,ds.\end{multline*}
\end{itemize}

Items $(i)-(iv)$ above, together with Theorem \ref{gedominatedconv}, ensure that $\big[\ln(t-s)-\psi(1)\big]f(s)$ is Bochner integrable and
\begin{equation}\label{ultimanovidade02}\lim_{n\rightarrow\infty}\int_{t_0}^t\left[\dfrac{(t-s)^{\tau_n}\big[\ln(t-s)-\psi(\tau_n+1)\big]}{\Gamma(\tau_n+1)}\right]f(s)\,ds
=\int_{t_0}^t\big[\ln(t-s)-\psi(1)\big]f(s)\,ds,\end{equation}
in the topology of $X$, for almost every $t\in[t_0,t_1]$.

Therefore, from \eqref{ultimanovidade01211}, \eqref{ultimanovidade01} and \eqref{ultimanovidade02} we deduce that
$$J^{1}_{t_0,t}Af(t)=\int_{t_0}^t\big[\ln(t-s)-\psi(1)\big]f(s)\,ds,$$
for almost every $t\in[t_0,t_1]$. Now observe that the above equality allows us to obtain
$$\int_{t_0}^t\ln(t-s)f(s)\,ds=\int_{t_0}^t\Big[Af(s)+\psi(1)f(s)\Big]\,ds,$$
for almost every $t\in[t_0,t_1]$. Since the right side of this identity is absolutely continuous with its derivative in $L^p(t_0,t_1;X)$, we conclude that the left side also has this property, as we wanted. Moreover, we obtain \eqref{charac}.

Conversely, let $f\in L^p(t_0,t_1;X)$ such that
$$\Phi(t):=\int_{t_0}^t\ln(t-s)f(s)\,ds,$$
is absolutely continuous in $[t_0,t_1]$ and its derivative belongs to $L^p(t_0,t_1;X)$. Observe that for any $0<\alpha_1<\alpha_2<\infty$, item $(ii)$ of Theorem \ref{mainsemigrupo} ensures that
$$J_{t_0,t}^{\alpha_2}f(t)-J_{t_0,t}^{\alpha_1}f(t)=A\underbrace{\left(\int_{\alpha_1}^{\alpha_2}J_{t_0,t}^\gamma f(t)\,d\gamma\right)}_{\in\, D(A)},$$
for almost every $t\in[t_0,t_1]$. Therefore, identity \eqref{charac} ensures that
\begin{multline}\label{lastequalityhope0666}J_{t_0,t}^{\alpha_2}f(t)-J_{t_0,t}^{\alpha_1}f(t)\\=-\psi(1)\left(\int_{\alpha_1}^{\alpha_2}J_{t_0,t}^\gamma f(t)\,d\gamma\right)+\dfrac{d}{dt}\left[\int_{t_0}^t\ln(t-s)\left(\int_{\alpha_1}^{\alpha_2}J_{t_0,s}^\gamma f(s)\,d\gamma\right)\,ds\right],\end{multline}
for almost every $t\in[t_0,t_1]$.

Note that Theorem \ref{FubiniFubini} ensures
\begin{equation*}\int_{t_0}^t\ln(t-s)\left(\int_{\alpha_1}^{\alpha_2}J_{t_0,s}^\gamma f(s)\,d\gamma\right)\,ds=\int_{\alpha_1}^{\alpha_2}\left(\int_{t_0}^{t}\ln(t-s)J_{t_0,s}^\gamma f(s)\,ds\right)\,d\gamma,\end{equation*}
for almost every $t\in[t_0,t_1]$, which we reinterpret as
\begin{multline}\label{lastequalityhope01}\int_{t_0}^t\ln(t-s)\left(\int_{\alpha_1}^{\alpha_2}J_{t_0,s}^\gamma f(s)\,d\gamma\right)\,ds\\=\int_{\alpha_1}^{\alpha_2}\left\{\int_{t_0}^{t}\left[\int_{t_0}^{s}\ln(t-s)\left(\dfrac{(s-w)^{\gamma-1}}{\Gamma(\gamma)}\right) f(w)\,dw\right]\,ds\right\}\,d\gamma,\end{multline}
for almost every $t\in[t_0,t_1]$.

However, since by changing the variables $w=s+t_0-h$ we deduce that
\begin{multline*}\int_{t_0}^{t}\left[\int_{t_0}^{s}\ln(t-s)\left(\dfrac{(s-w)^{\gamma-1}}{\Gamma(\gamma)}\right) f(w)\,dw\right]\,ds\\=\int_{t_0}^{t}\left[\int_{t_0}^{s}\ln(t-s)\left(\dfrac{(h-t_0)^{\gamma-1}}{\Gamma(\gamma)}\right) f(s+t_0-h)\,dh\right]\,ds,\end{multline*}
for almost every $t\in[t_0,t_1]$, we can apply Corollary \ref{FubiniFubini2} to obtain
\begin{multline*}\int_{t_0}^{t}\left[\int_{t_0}^{s}\ln(t-s)\left(\dfrac{(s-w)^{\gamma-1}}{\Gamma(\gamma)}\right) f(w)\,dw\right]\,ds\\=\dfrac{1}{\Gamma(\gamma)}\int_{t_0}^{t}(h-t_0)^{\gamma-1}\left[\int_{h}^{t}\ln(t-s)f(s+t_0-h)\,ds\right]\,dh,\end{multline*}
for almost every $t\in[t_0,t_1]$. Thus, by changing the variable $s=r+h-t_0$ we get
\begin{multline*}\int_{t_0}^{t}\left[\int_{t_0}^{s}\ln(t-s)\left(\dfrac{(s-w)^{\gamma-1}}{\Gamma(\gamma)}\right) f(w)\,dw\right]\,ds\\=\dfrac{1}{\Gamma(\gamma)}\int_{t_0}^{t}(h-t_0)^{\gamma-1}\left[\int_{t_0}^{t+t_0-h}\ln(t+t_0-r-h)f(r)\,dr\right]\,dh,\end{multline*}
for almost every $t\in[t_0,t_1]$, and finally by changing the variable $h=t+t_0-\tau$, we achieve the identity
\begin{multline}\label{ultma01}\int_{t_0}^{t}\left[\int_{t_0}^{s}\ln(t-s)\left(\dfrac{(s-w)^{\gamma-1}}{\Gamma(\gamma)}\right) f(w)\,dw\right]\,ds\\=\dfrac{1}{\Gamma(\gamma)}\int_{t_0}^{t}(t-\tau)^{\gamma-1}\left[\int_{t_0}^{\tau}\ln(\tau-r)f(r)\,dr\right]\,d\tau=J_{t_0,t}^\gamma\Phi(t),\end{multline}
for almost every $t\in[t_0,t_1]$. Thus, by replacing \eqref{ultma01} in \eqref{lastequalityhope01} we obtain
\begin{equation}\label{lastequalityhope02}\int_{t_0}^t\ln(t-s)\left(\int_{\alpha_1}^{\alpha_2}J_{t_0,s}^\gamma f(s)\,d\gamma\right)\,ds\\=\int_{\alpha_1}^{\alpha_2}J_{t_0,t}^\gamma\Phi(t)\,d\gamma,\end{equation}
for almost every $t\in[t_0,t_1]$.

Now, observe that Theorem \ref{leibnizint} and equation \eqref{lastequalityhope02} give us
\begin{equation*}\dfrac{d}{dt}\left[\int_{t_0}^t\ln(t-s)\left(\int_{\alpha_1}^{\alpha_2}J_{t_0,s}^\gamma f(s)\,d\gamma\right)\,ds\right]\\=\int_{\alpha_1}^{\alpha_2}\left[\dfrac{d}{dt}J_{t_0,t}^\gamma\Phi(t)\right]\,d\gamma,\end{equation*}
for almost every $t\in[t_0,t_1]$. Now the differentiability of $\Phi(t)$, the fact that $\Phi(t_0)=0$ and Theorem \ref{leibnizint}, allow us to rewrite the above identity as
\begin{equation*}\dfrac{d}{dt}\left[\int_{t_0}^t\ln(t-s)\left(\int_{\alpha_1}^{\alpha_2}J_{t_0,s}^\gamma f(s)\,d\gamma\right)\,ds\right]\\=\int_{\alpha_1}^{\alpha_2}J_{t_0,t}^\gamma\Phi^\prime(t)\,d\gamma,\end{equation*}
for almost every $t\in[t_0,t_1]$. Thus we may rewrite \eqref{lastequalityhope0666} as
$$J_{t_0,t}^{\alpha_2}f(t)-J_{t_0,t}^{\alpha_1}f(t)\\=\int_{\alpha_1}^{\alpha_2}J_{t_0,t}^\gamma \Big[-\psi(1)f(t)+\Phi^\prime(t)\Big]\,d\gamma,$$
for almost every $t\in[t_0,t_1]$. Making $\alpha_1\rightarrow0^+$, Theorem \ref{auxtextfunc2} help us to obtain
$$\dfrac{J_{t_0,t}^{\alpha_2}f(t)-f(t)}{\alpha_2}\\=\dfrac{\displaystyle\int_{0}^{\alpha_2}J_{t_0,t}^\gamma \Big[-\psi(1)f(t)+\Phi^\prime(t)\Big]\,d\gamma}{\alpha_2},$$
for almost every $t\in[t_0,t_1]$. Thus item $(i)$ of Theorem \ref{mainsemigrupo} and Theorem \ref{auxtextfunc2} ensure that
$$\lim_{\alpha_2\rightarrow0^+}\dfrac{J_{t_0,t}^{\alpha_2}f(t)-f(t)}{\alpha_2}\\=-\psi(1)f(t)+\Phi^\prime(t),$$
for almost every $t\in[t_0,t_1]$, i.e, that $f\in D(A)$.
\end{proof}

In order to answer the question proposed in the beginning of this subsection, we need the following technical lemma:

\begin{lemma}\label{digammafunction} If $n\in\mathbb{N}$, it holds that
$$\Gamma(n+1)J_{0,t}^n\ln(t)=t^n\Big[\ln(t)+\psi(1)-\psi(n+1)\Big],$$
for almost every $t\in[0,\infty)$. In the equation above $\psi(z)$ denotes the Digamma function.
\end{lemma}

\begin{proof} Observe that
$$\Gamma(n+1)J_{0,t}^n\ln(t)=n\int_{0}^t(t-s)^{n-1}\ln(s)\,ds,$$
for almost every $t\in[0,\infty)$. Therefore, by changing the variable $s=tr$ we obtain
\begin{multline*}\Gamma(n+1)J_{0,t}^n\ln(t)=n t^n\int_{0}^1(1-r)^{n-1}\ln(tr)\,dr=n t^n\int_{0}^1(1-r)^{n-1}\ln(r)\,dr\\+n t^n\ln(t)\int_{0}^1(1-r)^{n-1}\,dr,\end{multline*}
for almost every $t\in[0,\infty)$. Hence, we deduce
$$\Gamma(n+1)J_{0,t}^n\ln(t)= t^n\Big[\ln(t)+n\int_{0}^1(1-r)^{n-1}\ln(r)\,dr\Big],$$
for almost every $t\in[0,\infty)$. Finally, integration by parts and the recurrence relation of the Digamma function gives us
$$n\int_{0}^1(1-r)^{n-1}\ln(r)\,dr=-\sum_{k=1}^{n}\dfrac{1}{k}\qquad\textrm{and}\qquad \psi(n+1)-\psi(1)=\sum_{k=1}^{n}\dfrac{1}{k},$$
what allow us to conclude the arguments of this proof.
\end{proof}

\begin{theorem} Assume that $1\leq p\leq \infty$. If $A:D(A)\subset L^p(t_0,t_1;X)\rightarrow L^p(t_0,t_1;X)$ is the infinitesimal generator of the $C_0-$semigroup $\{J_{t_0,t}^\alpha:\alpha\geq0\}\subset \mathcal{L}(L^p(t_0,t_1;X))$, then $A:D(A)\subset L^p(t_0,t_1;X)\rightarrow L^p(t_0,t_1;X)$ is an unbounded operator.
\end{theorem}

\begin{proof} At first consider $p=\infty$. If we choose $x\in X$, with $\|x\|_X=1$, and define function $\phi\in L^\infty(t_0,t_1,X)$ by $\phi(t)=x$, we have that
$$\int_{t_0}^t\ln(t-s)\phi(s)\,ds=(t-t_0)\big[\ln(t-t_0)-1\big]x,$$
for almost every $t\in[t_0,t_1]$. Since the right side of the above equality is absolutely continuous we have
$$\dfrac{d}{dt}\left[\int_{t_0}^t\ln(t-s)\phi(s)\,ds\right]=\ln(t-t_0)x,$$
for almost every $t\in[t_0,t_1]$. However, since $\ln(t-t_0)x\not\in L^\infty(t_0,t_1,X)$, Theorem \ref{gerinf} ensures that $\phi\not\in D(A)$, when $D(A)$ is viewed as a domain in $L^\infty(t_0,t_1;X)$. This implies that $D(A)\subsetneq L^\infty(t_0,t_1;X)$, i.e., $A\not\in \mathcal{L}(L^\infty(t_0,t_1;X))$.

Now consider $1\leq p<\infty$. Recall from identity \eqref{charac} that $A=-\psi(1)Id+\widetilde{A}$, where $D(\widetilde{A})=D({A})$ and $\widetilde{A}:D(\widetilde{A})\subset L^p(t_0,t_1;X)\rightarrow L^p(t_0,t_1;X)$ is the linear operator given by
$$\widetilde{A}f(t)= \dfrac{d}{dt}\int_{t_0}^t\ln(t-s)f(s)\,ds.$$
Thus, it is clear that $\widetilde{A}$ is an unbounded operator if, and only if, $A$ is an unbounded operator. Therefore we shall dedicate our efforts to prove that $\widetilde{A}$ is an unbounded operator.

Let $x\in X$, with $\|x\|_X=1$, $n\in\mathbb{N}^*$ and consider sequence $\phi_{n}(t)=(t-t_0)^nx$. Observe that
$$\|\phi_n\|_{L^p(t_0,t_1;X)}=\left[\dfrac{(t_1-t_0)^{np+1}}{np+1}\right]^{1/p},$$
i.e., $\phi_n\in L^p(t_0,t_1;X)$, for any $n\in\mathbb{N}$.

On the other hand, observe that the regularity of $\phi_n(t)$ and Theorem \ref{leibnizint} allow us to deduce that
$$\dfrac{d}{dt}\int_{t_0}^t\ln(t-s)\phi_n(s)\,ds=\int_{t_0}^t\ln(t-s)\phi^\prime_n(s)\,ds,$$
for almost every $t\in[t_0,t_1]$. Therefore, by changing the variable $s=t-r$ we obtain
$$\dfrac{d}{dt}\int_{t_0}^t\ln(t-s)\phi_n(s)\,ds=xn\int_{0}^{t-t_0}(t-t_0-r)^{n-1}\ln(r)\,dr,$$
for almost every $t\in[t_0,t_1]$, which can be rewritten as
$$\dfrac{d}{dt}\int_{t_0}^t\ln(t-s)\phi_n(s)\,ds=x\Gamma(n+1)J_{0,\tau}^n\ln(\tau)\big|_{\tau=t-t_0},$$
for almost every $t\in[t_0,t_1]$. Thus, if we apply Lemma \ref{digammafunction} and change the variable $t=h+t_0$, we get
\begin{multline*}\left[\int_{t_0}^{t_1}\left\|\dfrac{d}{dt}\int_{t_0}^t\ln(t-s)\phi_n(s)\,ds\right\|_X^p\,dt\right]^{1/p}
\\=\left(\int_{0}^{t_1-t_0}\Big|h^n\big(\ln(h)+\psi(1)-\psi(n+1)\big)\Big|^p\,dh\right)^{1/p},\end{multline*}
where $\psi(z)$ denotes the Digamma function. Hence, since
$$\left(\int_{0}^{t_1-t_0}\Big|h^n\big(\ln(h)+\psi(1)-\psi(n+1)\big)\Big|^p\,dh\right)^{1/p}<\infty,$$
for any $n\in\mathbb{N}$, we deduce that $\phi_n\in D(\widetilde{A})$, for any $n\in\mathbb{N}$.

Now, since $\psi(n)\rightarrow\infty$, when $n\rightarrow\infty$, there exists $N_0\in\mathbb{N}$ such that
$$\psi(n+1)>\psi(1)+\ln(t_1-t_0),$$
for any $n\geq N_0$.
If this is the case, we know that for any $(t_1-t_0)/2\leq h\leq t_1-t_0$ it holds that
$$\ln(h)\leq\ln(t_1-t_0)<\psi(n+1)-\psi(1).$$
and therefore
\begin{multline*}\left\{\int_{0}^{t_1-t_0}\Big|h^n\big(\ln(h)+\psi(1)-\psi(n+1)\big)\Big|^p\,dh\right\}^{1/p}\\
\hspace*{-2cm}\geq\left\{\int_{(t_1-t_0)/2}^{t_1-t_0}\Big|h^n\big(\ln(h)+\psi(1)-\psi(n+1)\big)\Big|^p\,dh\right\}^{1/p}\\
\hspace*{2cm}\geq\big[\psi(n+1)-\psi(1)-\ln(t_1-t_0)\big]\left\{\int_{(t_1-t_0)/2}^{t_1-t_0}h^{np}\,dh\right\}^{1/p}\\
=\big[\psi(n+1)-\psi(1)-\ln(t_1-t_0)\big]\left[1-\dfrac{1}{2^{np+1}}\right]^{1/p}\left[\dfrac{(t_1-t_0)^{np+1}}{np+1}\right]^{1/p}.\end{multline*}

But then we deduce that
\begin{multline*}\lim_{n\rightarrow\infty}\dfrac{\|\widetilde{A}\phi_n\|_{L^p(t_0,t_1;X)}}{\|\phi_n\|_{L^p(t_0,t_1;X)}}\geq\lim_{n\rightarrow\infty}\big[\psi(n+1)-\psi(1)-\ln(t_1-t_0)\big]
\left[1-\dfrac{1}{2^{np+1}}\right]^{1/p}=\infty,\end{multline*}
i.e. $\widetilde{A}$ is an unbounded operator.
\end{proof}

A direct consequence of the above theorem and of item $(iv)$ of Theorem \ref{mainsemigrupo} is the following result.

\begin{corollary} The RL fractional integral operator from $L^p(t_0,t_1;X)$ into itself, is not uniformly continuous when $\alpha\rightarrow0^+$.
\end{corollary}

\subsection{The Unbounded Interval Case}

When $I=[t_0,\infty)$, the situation is completely different, as shown by the following result:

\begin{theorem} Let $1\leq p\leq\infty$ and $X$ a non trivial vector space. The RL fractional integral of order $\alpha>0$ is not well-defined as an operator from $L^{p}(t_0,\infty;X)$ into itself.
\end{theorem}

\begin{proof} Let $x\in X$, with $\|x\|_X=1$, consider $\alpha>0$ and define $\sigma_\infty:[t_0,\infty)\rightarrow X$ by $\sigma_\infty(t)=x$. Clearly $\sigma_\infty\in L^\infty(t_0,\infty;X)$. However, since
$$J_{t_0,t}^\alpha\sigma_\infty(t)=\dfrac{(t-t_0)^\alpha x}{\Gamma(\alpha+1)},$$
we conclude that $J_{t_0,t}^\alpha\sigma_\infty(t)$ does not belongs to $L^\infty(t_0,\infty;X)$.

Again, consider $x\in X$, with $\|x\|_X=1$. If $1\leq p<\infty$ and we define function $\sigma_p:[t_0,\infty)\rightarrow X$ by
$$\sigma_p(s)=\left\{\begin{array}{ll}0,&\textrm{if }t_0\leq s\leq t_0+1,\vspace*{0.2cm}\\(s-t_0)^{-(\alpha+(1/p))}x,&\textrm{if }s> t_0+1,\end{array}\right.$$
we observe that $\sigma_p\in L^p(t_0,\infty;X)$, since
$$\|\sigma_p\|_{L^p(t_0,\infty;X)}=\left[{\int_{t_0+1}^\infty \Big[(s-t_0)^{-[\alpha+(1/p)]}\Big]^p}\,ds\right]^{1/p}=\dfrac{1}{(\alpha p)^{1/p}}.$$
On the other hand, by changing variables $s-t_0=w(t-t_0)$, we have that
\begin{align*}J_{t_0,t}^\alpha \sigma_p(t)&=\left\{\begin{array}{ll}0,&\textrm{if }t_0\leq t\leq t_0+1,\vspace*{0.2cm}\\\dfrac{x}{\Gamma(\alpha)}\displaystyle\int_{t_0+1}^{t}{(t-s)^{\alpha-1}(s-t_0)^{-[\alpha+(1/p)]}}\,ds,&\textrm{if }t> t_0+1,\end{array}\right.\\
&=\left\{\begin{array}{ll}0,&\textrm{if }t_0\leq t\leq t_0+1,\vspace*{0.2cm}\\\dfrac{x(t-t_0)^{-(1/p)}}{\Gamma(\alpha)}\displaystyle\int_{1/(t-t_0)}^{1}{(1-w)^{\alpha-1}w^{-[\alpha+(1/p)]}}\,dw,&\textrm{if }t> t_0+1.\end{array}\right.\end{align*}
Again, by changing variables $t=t_0+(1/r)$, we deduce
\begin{align*}\|J_{t_0,t}^\alpha \sigma_p\|_{L^p(t_0,\infty;X)}&=\left\{\int_{t_0+1}^\infty{\left[\dfrac{(t-t_0)^{-(1/p)}}{\Gamma(\alpha)}\displaystyle\int_{1/(t-t_0)}^{1}{(1-w)^{\alpha-1}w^{-[\alpha+(1/p)]}}\,dw\right]^p}\,dt\right\}^{1/p}
\\&=\dfrac{1}{\Gamma(\alpha)}\left\{\int_{0}^1{r^{-1}\left[\displaystyle\int_{r}^{1}{(1-w)^{\alpha-1}w^{-[\alpha+(1/p)]}}\,dw\right]^p}\,dr\right\}^{1/p}
\\&\geq\dfrac{1}{\Gamma(\alpha)}\left\{\int_{0}^{1/2}{r^{-1}\left[\displaystyle\int_{r}^{1}{(1-w)^{\alpha-1}w^{-[\alpha+(1/p)]}}\,dw\right]^p}\,dr\right\}^{1/p}
\\&\geq\dfrac{1}{\Gamma(\alpha)}\left\{\int_{0}^{1/2}{r^{-1}\left[\displaystyle\int_{1/2}^{1}{(1-w)^{\alpha-1}w^{-[\alpha+(1/p)]}}\,dw\right]^p}\,dr\right\}^{1/p}
\\&=\dfrac{1}{\Gamma(\alpha)}\underbrace{\left\{\int_{0}^{1/2}{r^{-1}}\,dr\right\}^{1/p}}_{=\infty}\underbrace{\left[\displaystyle\int_{1/2}^{1}{(1-w)^{\alpha-1}w^{-[\alpha+(1/p)]}}\,dw\right]}_{<\infty}=\infty.\end{align*}
In other words, we have concluded that $J_{t_0,t}^\alpha \sigma_p(t)$ does not belongs to $L^p(t_0,\infty;X)$.
\end{proof}

\section{On the Sharp Constants to Lp-Lp Estimates}\label{sec4} To conclude the discussion about the RL fractional integral operator of order $\alpha>0$ from $L^p(t_0,t_1;X)$ into itself, we now discuss some properties about the constants $N>0$ that makes inequality (c.f. Theorem \ref{minkowskiseq})
\begin{equation}\label{naturalinclusion3}\|J_{t_0,s}^\alpha f\|_{L^p(t_0,t_1;X)}\leq N \|f\|_{L^p(t_0,t_1;X)},\end{equation}
holds for every $f\in L^p(t_0,t_1;X)$.

It is clear that the sharpest constant to \eqref{naturalinclusion3} is the norm of the operator $J_{t_0,t}^\alpha$ from $L^p(t_0,t_1;X)$ into itself, however to shorten the notation we shall refer to this value just as $N_{\alpha,p}$. In this way, we may compute $N_{\alpha,p}$ through the identity
\begin{equation}\label{nalphapsup}
  N_{\alpha,p}=\sup_{\phi\in L^{p}(t_0,t_1;X),\,\phi\not=0}{\left[\dfrac{\|J^\alpha_{t_0,t}\phi\|_{L^{p}(t_0,t_1;X)}}{\|\phi\|_{L^{p}(t_0,t_1;X)}}\right]}.
\end{equation}

Observe that Theorem \ref{minkowskiseq} gives us that $N_{\alpha,p}\leq{(t_1-t_0)^\alpha}/{\Gamma(\alpha+1)}.$  Furthermore, for the cases $p=1$ and $p=\infty$, it is not difficult to deduce the exactly value of $N_{\alpha,1}$ and $N_{\alpha,\infty}$. This is the subject of our next result.

\begin{theorem}\label{primeiroultimo} Assume that $\alpha>0$. If $p=1$ or $p=\infty$ it holds that
$$N_{\alpha,p}=\dfrac{(t_1-t_0)^\alpha}{\Gamma(\alpha+1)}.$$
\end{theorem}

\begin{proof} To prove the case when $p=\infty$, consider $x\in X$, with $\|x\|_X=1$, and define $\phi_\infty:[t_0,t_1]\rightarrow{X}$ by $\phi_\infty(t)=x$. With this, we have
$$\|\phi_\infty\|_{L^{\infty}(t_0,t_1;X)}=1\qquad\textrm{and}\qquad\|J_{t_0,t}^\alpha\phi_\infty\|_{L^{\infty}(t_0,t_1;X)}={(t_1-t_0)^\alpha}/{\Gamma(\alpha+1)}.$$
Equality \eqref{nalphapsup} and Theorem \ref{minkowskiseq} allow us to deduce that
$$\dfrac{(t_1-t_0)^\alpha}{\Gamma(\alpha+1)}=\dfrac{\|J^\alpha_{t_0,t}\phi_\infty\|_{L^{\infty}(t_0,t_1;X)}}{\|\phi_\infty\|_{L^{\infty}(t_0,t_1;X)}}\leq\dfrac{(t_1-t_0)^\alpha}{\Gamma(\alpha+1)},$$
and, therefore,  $N_{\alpha,\infty}={(t_1-t_0)^\alpha}/{\Gamma(\alpha+1)}.$

To prove the case $p=1$, let $x\in X$, with $\|x\|_X=1$, $\beta\in(0,1)$ and $\phi_\beta:[t_0,t_1]\rightarrow X$ be given by
$$\phi_\beta(t)=(t-t_0)^{-\beta}x.$$
Since we have
$$\|\phi_\beta\|_{L^{1}(t_0,t_1;X)}=\dfrac{(t_1-t_0)^{1-\beta}}{1-\beta}\quad\textrm{and}\quad
\|J_{t_0,t}^\alpha\phi_\beta\|_{L^{1}(t_0,t_1;X)}=\dfrac{\Gamma(1-\beta)(t_1-t_0)^{\alpha+1-\beta}}{\Gamma(\alpha+2-\beta)},$$
we obtain
$$\dfrac{\|J^\alpha_{t_0,t}\phi_\beta\|_{L^{1}(t_0,t_1;X)}}{\|\phi_\beta\|_{L^{1}(t_0,t_1;X)}}=\dfrac{\Gamma(2-\beta)(t_1-t_0)^\alpha}{\Gamma(\alpha+2-\beta)},$$
for any $\beta\in(0,1)$. But then, Theorem \ref{minkowskiseq} and the Squeeze Theorem ensures that
$$\lim_{\beta\rightarrow1^-}{\left[\dfrac{\Gamma(2-\beta)(t_1-t_0)^\alpha}{\Gamma(\alpha+2-\beta)}\right]}\leq\sup_{\phi\in L^{1}(t_0,t_1;X),\,\phi\not=0}{\left[\dfrac{\|J^\alpha_{t_0,t}\phi\|_{L^{1}(t_0,t_1;X)}}{\|\phi\|_{L^{1}(t_0,t_1;X)}}\right]}\leq\dfrac{(t_1-t_0)^\alpha}{\Gamma(\alpha+1)},$$
or in other words, $N_{\alpha,1}={(t_1-t_0)^\alpha}/{\Gamma(\alpha+1)}.$
\end{proof}

Since we are not able to give analogous results when $1<p<\infty$, in the next two subsections we present an interesting study that allows us to obtain more information about upper and lower bounds to $N_{\alpha,p}$.

\subsection{Maximizing functions.}

Last theorem proves that in $L^\infty(t_0,t_1;X)$ there exists a function $\phi_\infty(t)$ such that
$$N_{\alpha,\infty}={\dfrac{\|J^\alpha_{t_0,t}\phi_\infty\|_{L^{\infty}(t_0,t_1;X)}}{\|\phi_\infty\|_{L^{\infty}(t_0,t_1;X)}}}.$$
This insight made us question in which $L^p(t_0,t_1;X)$ space this behavior is repeated. Bearing this in mind, we introduce the notion of maximizing function, which is an idea inspired by Lieb with his impressive work \cite{Lieb1}.
\begin{definition} If $1\leq p\leq\infty$, we say that $f\in L^p(t_0,t_1;X)$, with $f\not=0$, is a maximizing function to inequality \eqref{naturalinclusion3} if
$$N_{\alpha,p}=\dfrac{\|J_{t_0,t}^\alpha f\|_{L^p(t_0,t_1;X)}}{\|f\|_{L^p(t_0,t_1;X)}}.$$
\end{definition}

Bellow we present the result that discuss the existence of maximizing functions to inequality \eqref{naturalinclusion3}.

\begin{theorem}\label{maxfunctions} Let $\alpha>0$. Then we have that:\vspace*{0,1cm}
\begin{itemize}
\item[(i)] If $p=\infty$, there exists a maximizing function to inequality \eqref{naturalinclusion3}.\vspace*{0,2cm}
\item[(ii)] If $1<p<\infty$, $X$ is reflexive and for any bounded set $F\subset L^p(t_0,t_1;X)$ it holds \eqref{novahip}, then there exists a maximizing function to inequality \eqref{naturalinclusion3}.\vspace*{0,2cm}
\item[(iii)] It does not exists a maximizing function to inequality \eqref{naturalinclusion3} when $p=1$.
\end{itemize}
\end{theorem}

\begin{proof} $(i)$ It follows directly from the first part of the proof of Theorem \ref{primeiroultimo}.\vspace*{0,2cm}

$(ii)$ Let $1<p<\infty$. By definition of $N_{\alpha,p}$, we may assume that there exists a sequence $\{f_n\}_{n=1}^\infty\subset L^{p}(t_0,t_1;X)$, with $\|f_n\|_{L^{p}(t_0,t_1;X)}=1$, such that
$$\lim_{n\rightarrow\infty}{\|J^\alpha_{t_0,t}f_n\|_{L^{p}(t_0,t_1;X)}}=N_{\alpha,p}.$$

Since $L^{p}(t_0,t_1;X)$ is reflexive (see item $(i)$ of Remark \ref{remarktestf}), taking into account Kakutani's Theorem (see \cite[Theorem 3.17]{Brezis1}), we deduce the existence of a subsequence $\{f_{n_k}\}_{k=1}^\infty\subset\{f_n\}_{n=1}^\infty$ and a limit function $f\in L^{p}(t_0,t_1;X)$, with $\|f\|_{L^{p}(t_0,t_1;X)}\leq1$, such that
$$f_{n_k}\rightarrow f,$$
in the weak-topology of $L^{p}(t_0,t_1;X)$.

On the other hand, since Theorem \ref{compactrieman} ensures that $J_{t_0,t}^\alpha$ is a compact operator, there should exists a subsequence $\{f_{n_{k_l}}\}_{l=1}^\infty\subset \{f_{n_k}\}_{k=1}^\infty$ and $g\in L^{p}(t_0,t_1;X)$ such that
$$J_{t_0,t}^\alpha f_{n_{k_l}}\rightarrow g,$$
in the strong-topology of $L^{p}(t_0,t_1;X)$. It is thus clear that $\|g\|_{L^{p}(t_0,t_1;X)}=N_{\alpha,p}$.

Again, the boundeness of $J_{t_0,t}^\alpha$ allows us to conclude that
$$J_{t_0,t}^\alpha f_{n_{k_l}}\rightarrow J_{t_0,t}^\alpha f,$$
in the weak-topology of $L^{p}(t_0,t_1;X)$. Since convergence in the strong-topology implies convergence in the weak-topology, the uniqueness of the limit allows us to conclude that
$$J_{t_0,t}^\alpha f(t)=g(t),$$
for almost every $t\in[t_0,t_1]$. Note that we still need to verify that $\|f\|_{L^p(t_0,t_1;X)}=1$ in order to $f(t)$ be a maximizing function. But, $\|f\|_{L^p(t_0,t_1;X)}=1$ since
$$N_{\alpha,p}=\|J_{t_0,t}^\alpha f\|_{L^p(t_0,t_1;X)}\leq N_{\alpha,p}\|f\|_{L^p(t_0,t_1;X)}\Longrightarrow 1\leq\|f\|_{L^p(t_0,t_1;X)}.$$

$(iii)$ Assume that there exists $f\in L^1(t_0,t_1;X)$ such that
$$N_{\alpha,1}=\dfrac{\|J_{t_0,t}^\alpha f\|_{L^1(t_0,t_1;X)}}{\|f\|_{L^1(t_0,t_1;X)}}.$$
Following the changing of variables done in the proof of Theorem \ref{continuity} we obtain the inequality
$$\int_{t_0}^{t_1}{\left\|J_{t_0,s}^\alpha f(s)\right\|_X}\,ds\leq \dfrac{1}{\Gamma(\alpha)} \int_{0}^{t_1-t_0}h^{\alpha-1}\left[\int_{t_0}^{t_1-h}\|f(l)\|_X\,dl\right]\,dh.$$

Since Theorem \ref{primeiroultimo} ensures that
$$N_{\alpha,1}=\dfrac{(t_1-t_0)^\alpha}{\Gamma(\alpha+1)},$$
we conclude that
\begin{equation}\label{fimdas1}\dfrac{(t_1-t_0)^\alpha}{\Gamma(\alpha+1)}=\dfrac{\|J_{t_0,t}^\alpha f\|_{L^1(t_0,t_1;X)}}{\|f\|_{L^1(t_0,t_1;X)}}\leq \dfrac{1}{\Gamma(\alpha)} \int_{0}^{t_1-t_0}h^{\alpha-1}\underbrace{\left[\dfrac{\int_{t_0}^{t_1-h}\|f(l)\|_X\,dl}{\|f\|_{L^1(t_0,t_1;X)}}\right]}_{=\phi(h)}\,dh.\end{equation}

Since $\phi(h)$ is monotonically decreasing and $h^{\alpha-1}$ is integrable in $(0,t_1-t_0)$, Theorem \ref{secondmainvalue} ensures the existence of $\xi\in(0,t_1-t_0)$ such that
\begin{multline}\label{fimdas2}\int_{0}^{t_1-t_0}h^{\alpha-1}{\phi(h)}\,dh=\phi(0^+)\left[\int_{0}^\xi h^{\alpha-1}\,dh\right]+\phi([t_1-t_0]^-)\left[\int_{\xi}^{t_1-t_0}h^{\alpha-1}\,dh\right]\\
=\int_{0}^\xi h^{\alpha-1}\,dh=\dfrac{\xi^\alpha}{\alpha}\end{multline}

Therefore, from \eqref{fimdas1} and \eqref{fimdas2} we deduce that
$$\dfrac{(t_1-t_0)^\alpha}{\Gamma(\alpha+1)}\leq \dfrac{\xi^\alpha}{\Gamma(\alpha+1)}<\dfrac{{(t_1-t_0)}^\alpha}{\Gamma(\alpha+1)},$$
what is a contradiction. This completes the proof.
\end{proof}

When $X$ is finite dimensional, since $X$ is reflexive and $J_{t_0,t}^\alpha$ is compact, we obtain directly from the above Theorem that

\begin{corollary} Let $\alpha>0$ and assume that $X$ is a finite dimensional Banach space. Then we have that:\vspace*{0,1cm}
\begin{itemize}
\item[(i)] If $1<p\leq\infty$ there exists a maximizing function to inequality \eqref{naturalinclusion3}.\vspace*{0,2cm}
\item[(ii)] It does not exists a maximizing function to inequality \eqref{naturalinclusion3} when $p=1$.
\end{itemize}
\end{corollary}


\subsection{Lower and uper bounds to $N_{\alpha,p}$}

We would like to point out that even if the exact value of $N_{\alpha,p}$, when $1<p<\infty$, remains unclear to us, at least on reflexive Banach spaces, we can guarantee that this value is strictly lower than $(t_1-t_0)^{\alpha}/\Gamma(\alpha+1)$. This is the subject of our next Theorem. 

\begin{theorem}\label{maxfunctions2} Consider $\alpha>0$, $1<p<\infty$ and $X$ a reflexive Banach space. Then we have that $N_{\alpha,p}<(t_1-t_0)^{\alpha}/\Gamma(\alpha+1)$.
\end{theorem}

\begin{proof} Let $f_p\in L^p(t_0,t_1;X)$ be the maximizing function to inequality \eqref{naturalinclusion3}, ensured by Theorem \ref{maxfunctions}. If we assume that $N_{\alpha,p}=(t_1-t_0)^{\alpha}/\Gamma(\alpha+1)$, like it was observed in the proof of item $(iii)$ of Theorem \ref{maxfunctions}, we deduce that
\begin{equation*}\dfrac{(t_1-t_0)^\alpha}{\Gamma(\alpha+1)}=\left[\dfrac{\|J_{t_0,t}^\alpha f_p\|_{L^p(t_0,t_1;X)}}{\|f_p\|_{L^p(t_0,t_1;X)}}\right]\leq \dfrac{1}{\Gamma(\alpha)} \int_{0}^{t_1-t_0}h^{\alpha-1}\underbrace{\left\{\dfrac{\left[\int_{t_0}^{t_1-h}\|f_p(l)\|^p_X\,dl\right]^{1/p}}{\|f_p\|_{L^p(t_0,t_1;X)}}\right\}}_{=\phi_p(h)}\,dh.\end{equation*}

The remaining of the proof follows exactly as done in item $(iii)$ of Theorem \ref{maxfunctions}.

\end{proof}

\begin{corollary} Consider $\alpha>0$, $1<p<\infty$ and $X$ a finite dimensional Banach space. Then we have that $N_{\alpha,p}<(t_1-t_0)^{\alpha}/\Gamma(\alpha+1)$.
\end{corollary}

\begin{proof} It follows directly from Theorem \ref{maxfunctions2} and the fact that $X$ is a finite dimensional Banach space.
\end{proof}

We dedicate the remainder of this section to obtain lower and (better) upper bounds to the value $N_{\alpha,p}$, when $1<p<\infty$.

Let us begin with some auxiliary lemmas. Before that, we just emphasize that for the remainder of this section, $p^*$ is denoting the Holder conjugate exponent of $p$, i.e., $p^*=p/(p-1)$.

\begin{lemma}\label{firstlemaaux} Assume that $\alpha>0$ and $1<p<\infty$.
\begin{itemize}
\item[(i)] If $\alpha>1/p^*$, we have that:
$$\|J_{t_0,t}^\alpha f\|_{L^p(t_0,t_1;X)}\leq \left\{\dfrac{(t_1-t_0)^\alpha}{\Gamma(\alpha)\big[(\alpha-1)p+1\big]^{1/p}\big(\alpha p^*\big)^{1/p^*}}\right\}\|f\|_{L^p(t_0,t_1;X)},$$
for every $f\in L^p(t_0,t_1;X)$. \vspace{0.2cm}
\item[(ii)] If $\alpha>1/p$, we have that:
$$\|J_{t_0,t}^\alpha f\|_{L^p(t_0,t_1;X)}\leq \left\{\dfrac{(t_1-t_0)^\alpha}{\Gamma(\alpha)\big[(\alpha-1)p^*+1\big]^{1/p^*}\big(\alpha p\big)^{1/p}}\right\}\|f\|_{L^p(t_0,t_1;X)},$$
for every $f\in L^p(t_0,t_1;X)$.
\end{itemize}
\end{lemma}

\begin{proof}$(i)$ Let $f\in L^p(t_0,t_1;X)$. Observe that Corollary \ref{minkowskiminkowski22} ensures
\begin{multline*}\|J_{t_0,t}^\alpha f\|_{L^p(t_0,t_1;X)}\leq\dfrac{1}{\Gamma(\alpha)}\left[\int_{t_0}^{t_1}\left[\int_{t_0}^t(t-s)^{\alpha-1}\|f(s)\|_X\,ds\right]^pdt\right]^{1/p}\\
\leq\dfrac{1}{\Gamma(\alpha)}\int_{t_0}^{t_1}\left[\int_{s}^{t_1}(t-s)^{(\alpha-1)p}\,dt\right]^{1/p}\|f(s)\|_X\,ds
\\=\left\{\dfrac{1}{\Gamma(\alpha)\big[(\alpha-1)p+1\big]^{1/p}}\right\}\int_{t_0}^{t_1}(t_1-s)^{\alpha-1+(1/p)}\|f(s)\|_X\,ds.\end{multline*}

Finally, Holder's inequality gives us
$$\|J_{t_0,t}^\alpha f\|_{L^p(t_0,t_1;X)}\leq \left\{\dfrac{(t_1-t_0)^\alpha}{\Gamma(\alpha)\big[(\alpha-1)p+1\big]^{1/p}\big(\alpha p^*\big)^{1/p^*}}\right\}\|f\|_{L^p(t_0,t_1;X)}.$$
$(ii)$  Let $f\in L^p(t_0,t_1;X)$. Observe that Holder's inequality ensures
\begin{multline*}\|J_{t_0,t}^\alpha f(t)\|_{X}\leq\dfrac{1}{\Gamma(\alpha)}\int_{t_0}^t(t-s)^{\alpha-1}\|f(s)\|_X\,ds\\\leq\dfrac{1}{\Gamma(\alpha)}\left[\int_{t_0}^{t}(t-s)^{(\alpha-1)p^*}\,ds\right]^{1/p^*}\left[\int_{t_0}^{t}{\|f(s)\|_X^p}\,ds\right]^{1/p}
\\=\left\{\dfrac{(t-t_0)^{\alpha-1+(1/p^*)}}{\Gamma(\alpha)\big[(\alpha-1)p^*+1\big]^{1/p^*}}\right\}\|f\|_{L^p(t_0,t_1;X)}.\end{multline*}

But then,
$$\|J_{t_0,t}^\alpha f\|_{L^p(t_0,t_1;X)}\leq \left\{\dfrac{(t_1-t_0)^\alpha}{\Gamma(\alpha)\big[(\alpha-1)p^*+1\big]^{1/p^*}\big(\alpha p\big)^{1/p}}\right\}\|f\|_{L^p(t_0,t_1;X)}.$$
\end{proof}

\begin{remark}\label{remarkcomplementary}
  When $\alpha\geq1$ and $1< p<\infty$, it always holds that $\alpha>1/p$ and $\alpha>1/p^*$. Thus, Lemma \ref{firstlemaaux} allows us to obtain the estimate\vspace*{0.2cm}
$$N_{\alpha,p}\leq \min{\left\{\dfrac{(t_1-t_0)^\alpha}{\Gamma(\alpha)\big[(\alpha-1)p+1\big]^{1/p}\big(\alpha p^*\big)^{1/p^*}},\dfrac{(t_1-t_0)^\alpha}{\Gamma(\alpha)\big[(\alpha-1)p^*+1\big]^{1/p^*}\big(\alpha p\big)^{1/p}}\right\}}.\vspace*{0.3cm}$$
\end{remark}

It is clear that this minimum, if less than $(t_1-t_0)^\alpha/\Gamma(\alpha+1)$, is a better upper bound to $N_{\alpha,p}$. To prove our next theorem we first present the following auxiliary lemma:

\begin{lemma}\label{secondlemaaux} Consider $1<p<\infty$ and the function $\xi:(0,1]\rightarrow\mathbb{R}$ given by
\begin{equation}\label{funcaoxi}\xi(t):=\dfrac{1+(p-1)(t-1)}{t^{p-1}}.\end{equation}
Then it holds that:
\begin{itemize}
\item[(i)] If $1<p<2$, \,$\xi(t)>1,\forall t\in(0,1);$
\item[(ii)] If $p=2$, \,$\xi(t)=1,\forall t\in(0,1);$
\item[(iii)] If $p>2$, \,$\xi(t)<1,\forall t\in(0,1).$
\end{itemize}
\end{lemma}

\begin{proof}
Just observe that $\xi^\prime(t)=t^{-p}\big(p-1\big)\big(p-2\big)\big(1-t\big)$, and $\xi(1)=1$.

\end{proof}

\begin{theorem} Assume that $\alpha\geq1$ and $1< p<\infty$.
\begin{itemize}
\item[(i)] If $1<p\leq2$ then \vspace*{0.2cm}
\begin{equation}\label{estimativefinall01}N_{\alpha,p}\leq\dfrac{(t_1-t_0)^\alpha}{\Gamma(\alpha)\big[(\alpha-1)p+1\big]^{1/p}\big(\alpha p^*\big)^{1/p^*}}<\dfrac{(t_1-t_0)^\alpha}{\Gamma(\alpha+1)}.\vspace*{0.3cm}\end{equation}
\item[(ii)] For $p>2$ we have\vspace*{0.2cm}
\begin{equation}\label{estimativefinall02}N_{\alpha,p}\leq\dfrac{(t_1-t_0)^\alpha}{\Gamma(\alpha)\big[(\alpha-1)p^*+1\big]^{1/p^*}\big(\alpha p\big)^{1/p}}<\dfrac{(t_1-t_0)^\alpha}{\Gamma(\alpha+1)}.\vspace*{0.3cm}\end{equation}
\end{itemize}
\end{theorem}

\begin{proof} $(i)$

Observe that inequality
\begin{equation}\label{secondlemaauxaux121}\dfrac{(t_1-t_0)^\alpha}{\Gamma(\alpha)\big[(\alpha-1)p+1\big]^{1/p}\big(\alpha p^*\big)^{1/p^*}}\leq\dfrac{(t_1-t_0)^\alpha}{\Gamma(\alpha)\big[(\alpha-1)p^*+1\big]^{1/p^*}\big(\alpha p\big)^{1/p}}\end{equation}
is equivalent to
\begin{equation}\label{secondlemaauxaux}1\leq\dfrac{1-(p-1)[1/(\alpha p)]}{\big\{1-[1/(\alpha p)]\big\}^{p-1}}.\end{equation}

Since $1-\big[1/(\alpha p)\big]\in(0,1)$, \eqref{secondlemaauxaux} is equivalent to
$$1\leq\xi\Big(1-\big[1/(\alpha p)\big]\Big),$$
where $\xi(t)$ is defined in \eqref{funcaoxi}. Since $1<p\leq2$, Lemma \ref{secondlemaaux} guarantees that inequality \eqref{secondlemaauxaux} holds, what implies that the first part of inequality \eqref{estimativefinall01} holds. It is also important to emphasize that the equality in \eqref{secondlemaauxaux}, and consequently in \eqref{secondlemaauxaux121}, occurs only when $p=2$.

To verify the second part of inequality \eqref{estimativefinall01}, notice that
$$\alpha\geq1>\dfrac{p^{p-1}(p-1)}{p^p-(p-1)^{p-1}}\Longrightarrow \big[(\alpha-1)p+1\big]^{1/p}(\alpha p^*)^{1/p^*}>\alpha,$$
what ensures the inequality
$$\dfrac{(t_1-t_0)^\alpha}{\Gamma(\alpha)\big[(\alpha-1)p+1\big]^{1/p}\big(\alpha p^*\big)^{1/p^*}}<\dfrac{(t_1-t_0)^\alpha}{\Gamma(\alpha+1)}.\vspace*{0.2cm}$$

$(ii)$
Like before, since inequality
$$\dfrac{(t_1-t_0)^\alpha}{\Gamma(\alpha)\big[(\alpha-1)p+1\big]^{1/p}\big(\alpha p^*\big)^{1/p^*}}>\dfrac{(t_1-t_0)^\alpha}{\Gamma(\alpha)\big[(\alpha-1)p^*+1\big]^{1/p^*}\big(\alpha p\big)^{1/p}}$$
is equivalent to
\begin{equation}\label{secondlemaauxaux2}1>\dfrac{1-(p-1)[1/(\alpha p)]}{\big\{1-[1/(\alpha p)]\big\}^{p-1}},\end{equation}
with $1-\big[1/(\alpha p)\big]\in(0,1)$ and $p>2$, Lemma \ref{secondlemaaux} guarantees that inequality \eqref{secondlemaauxaux2} holds, what implies that the first part of inequality \eqref{estimativefinall02} holds.

Now observe that
$$\alpha\geq1>\dfrac{p^{1/(p-1)}}{p^{p/(p-1)}-(p-1)}\Longrightarrow \big[(\alpha-1)p^*+1\big]^{1/p^*}(\alpha p)^{1/p}>\alpha,$$
what ensures that the second part of inequality \eqref{estimativefinall02} holds.

\end{proof}

Let us now discuss the case $0<\alpha<1$, which demands more attention. To begin, consider functions $\eta,\theta:(1,\infty)\rightarrow(0,\infty)$ given by
\begin{equation}\label{funcimpo}\eta(t)=\dfrac{t^{t-1}(t-1)}{t^t-(t-1)^{t-1}}\qquad\textrm{and}\qquad\theta(t)=\dfrac{t^{1/(t-1)}}{t^{t/(t-1)}-(t-1)}.\end{equation}

Since $\eta(t)$ is a continuous increasing function that satisfies
$$\lim_{t\rightarrow1^+}\eta(t)=0\quad\textrm{and}\quad\lim_{t\rightarrow\infty}\eta(t)=1,$$
while $\theta(t)$ is a continuous decreasing function that satisfies
$$\lim_{t\rightarrow1^+}\theta(t)=1\quad\textrm{and}\quad\lim_{t\rightarrow\infty}\theta(t)=0,$$
we may consider $p_{1,\alpha}$ and $p_{2,\alpha}$ the respective unique solutions in $(1,\infty)$ of the equations
\begin{equation}\label{impoconst}\alpha=\eta(p_{1,\alpha})\qquad\textrm{and}\qquad \alpha=\theta(p_{2,\alpha}).\end{equation}

The above considerations are enough for us to present our next result.

\begin{theorem} Consider $0< \alpha<1$ and $1< p<\infty$.
\begin{itemize}
\item[(i)]  If $\alpha>1/p^*$, then we have that\vspace*{0.2cm}
\begin{equation}\label{estimativefinal16a}N_{\alpha,p}\leq\dfrac{(t_1-t_0)^\alpha}{\Gamma(\alpha)\big[(\alpha-1)p+1\big]^{1/p}\big(\alpha p^*\big)^{1/p^*}}.\vspace*{0.3cm}\end{equation}
Moreover, if $p<p_{1,\alpha}$, where $p_{1,\alpha}$ is given in \eqref{impoconst}, we have that
\begin{equation}\label{estimativefinal16b}\dfrac{(t_1-t_0)^\alpha}{\Gamma(\alpha)\big[(\alpha-1)p+1\big]^{1/p}\big(\alpha p^*\big)^{1/p^*}}<\dfrac{(t_1-t_0)^\alpha}{\Gamma(\alpha+1)}.\vspace*{0.3cm}\end{equation}
Otherwise, i.e., when $p\geq p_{1,\alpha}$, we cannot improve the boundeness given by Theorem \ref{maxfunctions2} with Lemma \ref{firstlemaaux}.\vspace*{0.3cm}
\item[(ii)]  If $\alpha>1/p$, then we have that\vspace*{0.2cm}
\begin{equation}\label{estimativefinal17a}N_{\alpha,p}\leq\dfrac{(t_1-t_0)^\alpha}{\Gamma(\alpha)\big[(\alpha-1)p^*+1\big]^{1/p^*}\big(\alpha p\big)^{1/p}}.\vspace*{0.3cm}\end{equation}
Moreover, if $p>p_{2,\alpha}$, where $p_{2,\alpha}$ is given in \eqref{impoconst}, we have that
\begin{equation}\label{estimativefinal17b}\dfrac{(t_1-t_0)^\alpha}{\Gamma(\alpha)\big[(\alpha-1)p^*+1\big]^{1/p^*}\big(\alpha p\big)^{1/p}}<\dfrac{(t_1-t_0)^\alpha}{\Gamma(\alpha+1)}.\vspace*{0.3cm}\end{equation}
Otherwise, i.e., when $p\leq p_{2,\alpha}$, we cannot improve the boundeness given by Theorem \ref{maxfunctions2} with Lemma \ref{firstlemaaux}.\vspace*{0.3cm}
\end{itemize}
\end{theorem}

\begin{proof}
When $\alpha>1/p^*$, Lemma \ref{firstlemaaux} ensures that inequality \eqref{estimativefinal16a} holds, and when $\alpha>1/p$ it guarantees that inequality \eqref{estimativefinal17a} holds.

Now, to prove the second part of item $(i)$, recall that $\eta(t)$, given in \eqref{funcimpo}, is an increasing function and that $p<p_{1,\alpha}$. Therefore, we have that
$$\alpha=\eta(p_{1,\alpha})>\eta(p)=\dfrac{p^{p-1}(p-1)}{p^p-(p-1)^{p-1}}\Longrightarrow \big[(\alpha-1)p+1\big]^{1/p}(\alpha p^*)^{1/p^*}>\alpha,$$
i.e., inequality \eqref{estimativefinal16b} holds.

Finally, to prove the second part of item $(ii)$, recall that $\theta(t)$, given in \eqref{funcimpo}, is a decreasing function and that $p>p_{2,\alpha}$. Therefore, we have that
$$\alpha=\theta(p_{2,\alpha})>\theta(p)=\dfrac{p^{1/(p-1)}}{p^{p/(p-1)}-(p-1)}\Longrightarrow \big[(\alpha-1)p^*+1\big]^{1/p^*}(\alpha p)^{1/p}>\alpha,$$
i.e., inequality \eqref{estimativefinal17b} holds\end{proof}

To end this section, let us discuss the existence of a lower bound to the constant $N_{\alpha,p}$, when $1<p<\infty$. We follow one of the ideas presented in the proof of Theorem \ref{primeiroultimo}.

Let $x\in X$, $\beta\in[0,1/p)$ and $\phi_\beta:[t_0,t_1]\rightarrow X$ be given by
$$\phi_\beta(t)=(t-t_0)^{-\beta}x.$$
Since we have
\begin{multline*}\|\phi_\beta\|_{L^{p}(t_0,t_1;X)}=\dfrac{(t_1-t_0)^{(1/p)-\beta}\|x\|_X}{(1-\beta p)^{1/p}}\quad\textrm{and}
\quad\\\|J_{t_0,t}^\alpha\phi_\beta\|_{L^{p}(t_0,t_1;X)}=\dfrac{\Gamma(1-\beta)(t_1-t_0)^{\alpha+(1/p)-\beta}\|x\|_X}{\Gamma(\alpha+1-\beta)[(\alpha-\beta)p+1]^{1/p}},\end{multline*}
we deduce by inequality \ref{naturalinclusion3} that
\begin{equation}\label{finalcomp}\hspace*{0.2cm}\dfrac{(t_1-t_0)^\alpha\,\Gamma(1-\beta)(1-\beta p)^{1/p}}{\Gamma(\alpha+1-\beta)[(\alpha-\beta)p+1]^{1/p}}=\dfrac{\|J^\alpha_{t_0,t}\phi_\beta\|_{L^{p}(t_0,t_1;X)}}{\|\phi_\beta\|_{L^{p}(t_0,t_1;X)}}\leq N_{\alpha,p},\end{equation}
for any $\beta\in[0,1/p)$.

By considering function $\zeta:[0,1/p)\rightarrow\mathbb{R}$ given by\vspace*{0.2cm}
$$\zeta(t)=\dfrac{\Gamma(1-t)(1-t p)^{1/p}}{\Gamma(\alpha+1-t)[(\alpha-t)p+1]^{1/p}},\vspace*{0.2cm}$$
we may check numerically that there exists $\beta_p\in [0,1/p)$ that is the maximum value of this function. Analytically it is complicated to determine $\beta_p$ because it would be the unique solution of the equation\vspace*{0.2cm}
$$g(t)\Big\{\alpha p+(1-tp)\big[(\alpha-t)p+1\big]\Big[\psi(1-t)-\psi(\alpha+1-t)\Big]\Big\}=\zeta^\prime(t)=0,\vspace*{0.2cm}$$
where $\psi(t)$ is the Digamma function and $g:[0,1/p)\rightarrow\mathbb{R}$ is given by\vspace*{0.2cm}
$$g(t)=-\left[\dfrac{\Gamma(1-t)(1-tp)^{(1/p)-1}}{\Gamma(\alpha+1-t)\big[(\alpha-t)p+1\big]^{(1/p)+1}}\right].\vspace*{0.2cm}$$

Therefore, we choose $\beta=0$ in \eqref{finalcomp} as a lower (computable) bound to $N_{\alpha,p}$. This discussion can be described by the following result.

\begin{theorem} Let $\alpha>0$ and $1<p<\infty$. Then \vspace*{0.2cm}
$$\dfrac{(t_1-t_0)^\alpha}{\Gamma(\alpha+1)[\alpha p+1]^{1/p}}\leq N_{\alpha,p}.$$
\end{theorem}

\section{Compilation of some classical results}\label{sec5}

Let us summarize some well-known results that are used throughout this manuscript. Besides, since we use several classical theorems adapted to measurable vector-valued functions, we find useful to state and make bibliographical references to these theorems in order to keep this text self contained.

The first one is the classical Riesz-Thorin interpolation of real functions (see \cite[Theorem 6.27]{Fol01} for details on the proof).

\begin{theorem}[Riesz-Thorin Interpolation]\label{riezthorin} Consider $1\leq p_\theta,p_1,p_2,q_\theta,q_1,q_2\leq\infty$ and $\theta\in(0,1)$ satisfying the identities
$$\dfrac{1}{p_\theta}=\dfrac{(1-\theta)}{p_1}+\dfrac{\theta}{p_2}\qquad\textrm{and}\qquad\dfrac{1}{q_\theta}=\dfrac{(1-\theta)}{q_1}+\dfrac{\theta}{q_2}.$$
If $T$ is a linear operator that maps the space $L^{p_1}(t_0,t_1;\mathbb{R})+L^{p_2}(t_0,t_1;\mathbb{R})$ into the space $L^{q_1}(t_0,t_1;\mathbb{R})+L^{q_2}(t_0,t_1;\mathbb{R})$ and satisfies
$$\|Tf\|_{L^{q_1}(t_0,t_1;\mathbb{R})}\leq M_1\|f\|_{L^{p_1}(t_0,t_1;\mathbb{R})},$$
for any $f\in L^{p_1}(t_0,t_1;\mathbb{R})$ and
$$\|Tf\|_{L^{q_2}(t_0,t_1;\mathbb{R})}\leq M_2\|f\|_{L^{p_2}(t_0,t_1;\mathbb{R})},$$
for any $f\in L^{p_2}(t_0,t_1;\mathbb{R})$, then $T$ is a linear operator that maps $L^{p_\theta}(t_0,t_1;\mathbb{R})$ into $L^{q_\theta}(t_0,t_1;\mathbb{R})$ and satisfies
$$\|Tf\|_{L^{q_\theta}(t_0,t_1;\mathbb{R})}\leq M_1^{1-\theta}M_2^\theta\|f\|_{L^{p_\theta}(t_0,t_1;\mathbb{R})},$$
for any $f\in L^{p_\theta}(t_0,t_1;\mathbb{R})$.
\end{theorem}

We also recall Minkowski's inequality to integrals, which is a classical result that can be found in \cite[Theorem 202]{HaLiPo1}.

\begin{theorem}[Minkowski's Inequality to Integrals]\label{minkowski} Let $f:[t_0,t_1]\times[t_0',t_1']\rightarrow\mathbb{R}$ be a Lebesgue measurable function and assume that $1\leq p<\infty$. Then
$$\left[\int_{t_0'}^{t_1'}{\left|\int_{t_0}^{t_1}{f(s,r)}\,ds\right|^p}\,dr\right]^{1/p}
\leq\int_{t_0}^{t_1}{\left[\int_{t_0'}^{t_1'}{\big|f(s,r)\big|^p}\,dr\right]^{1/p}}\,ds.$$
\end{theorem}

Since we recursively apply two specific consequences of Minkowski's inequality to integrals, which are not in the form presented above, we state them bellow.

\begin{corollary}\label{minkowskiminkowski22} If $f:[t_0,t_1]\rightarrow X$ is a Bochner integrable function, $g:\mathbb{R}\rightarrow[0,\infty)$ is a locally Lebesgue integrable function and $1\leq p<\infty$, then
\begin{equation*}\left[\int_{t_0}^{t_1}{\left[\int_{t_0}^{r}{g(r-s)\left\|f(s)\right\|_X}\,ds\right]^p}\,dr\right]^{1/p}\leq
\int_{t_0}^{t_1}{\left[\int_{s}^{t_1}{\big[g(r-s)\big]^p}\,dr\right]^{1/p}}\|f(s)\|_X\,ds.\end{equation*}
\end{corollary}

\begin{corollary}\label{minkowskiminkowski} Let $f:[t_0,t_1]\rightarrow X$ be a Bochner integrable function and $g:\mathbb{R}\rightarrow[0,\infty)$ a locally Lebesgue integrable function. If $1\leq p<\infty$ and $0\leq h<t_1-t_0$, we have that
\begin{multline*}\left[\int_{0}^{t_1-t_0-h}\left[\int_{0}^{r}g(s)\|f(r+t_0-s)\|_X\,ds\right]^p\,dr\right]^{1/p}
\\\leq\int_{0}^{t_1-t_0-h}g(s)\left[\int_{s}^{t_1-t_0-h}\|f(r+t_0-s)\|^p_X\,dr\right]^{1/p}\,ds.\end{multline*}
\end{corollary}

Let us recall an integral calculus theorem, which was proved by Hobson in \cite{Hob1}.

\begin{theorem}[Second Mean Value Theorem]\label{secondmainvalue} If $f:[t_0,t_1]\rightarrow[0,\infty)$ is a monotonic function (not necessarily decreasing or increasing) and $g:[t_0,t_1]\rightarrow\mathbb{R}$ is an integrable function, then there exists $\xi\in(t_0,t_1)$ such that
$$\int_{t_0}^{t_1}f(s)g(s)\,ds=f(t_0^+)\left[\int_{t_0}^\xi g(s)\,ds\right]+f(t_1^-)\left[\int_{\xi}^{t_1}g(s)\,ds\right].$$
Here $f(t_0^+)$ stands for $\lim_{t\rightarrow t_0^+}f(t)$ and $f(t_1^-)$ stands for $\lim_{t\rightarrow t_1^-}f(t)$.
\end{theorem}

We continue by stating Fubini's theorem (see \cite[Chapter X - Theorem 2]{Mik1} for details on the proof).

\begin{theorem}[Fubini's Theorem]\label{FubiniFubini} If $f:[t_0,t_1]\times[t_0',t_1']\rightarrow X$ is a Bochner integrable function, then function
$$[t_0,t_1]\ni s\mapsto\int_{t_0'}^{t_1'}{f(s,r)}\,dr$$
exists (the norm of the integral is finite) almost everywhere and is Bochner integrable on $[t_0,t_1]$. Similarly, the function
$$[t_0',t_1']\ni r\mapsto\int_{t_0}^{t_1}{f(s,r)}\,ds$$
exists almost everywhere and is Bochner integrable on $[t_0',t_1']$. Moreover,
$$\int_{t_0'}^{t_1'}{\int_{t_0}^{t_1}{f(s,r)}\,ds}\,dr
=\int_{t_0'}^{t_1'}{\left[\int_{t_0}^{t_1}{f(s,r)}\,ds\right]}\,dr=\int_{t_0}^{t_1}{\left[\int_{t_0'}^{t_1'}{f(s,r)}\,dr\right]}\,ds.$$
\end{theorem}

For the same reason pointed out to introduce Corollaries \ref{minkowskiminkowski22} and \ref{minkowskiminkowski}, here we also present an auxiliary result.

\begin{corollary}\label{FubiniFubini2} If $f:[t_0,t_1]\rightarrow X$ is a Bochner integrable function and $g,h:\mathbb{R}\rightarrow\mathbb{R}$ are locally Lebesgue integrable functions. Then
\begin{equation*}\int_{t_0}^{t_1}{\left[\int_{t_0}^{r}{g(r-s)h(r)f(s)}\,ds\right]}\,dr=\int_{t_0}^{t_1}{\left[\int_{s}^{t_1}{g(r-s)h(r)}\,dr\right]}f(s)\,ds.\end{equation*}
\end{corollary}

An important result that we use recursively is the Generalized Dominated Convergence Theorem to Bochner integrable functions. The proof of this result is analogous to its classical version.

\begin{theorem}[Bochner - Generalized Dominated Convergence Theorem]\label{gedominatedconv} Assume that $f_n:[t_0,t_1]\rightarrow X$ is a sequence of Bochner measurable functions and $g_n:[t_0,t_1]\rightarrow \mathbb{R}$ is a sequence of Lebesgue integrable functions that satisfies:
\begin{itemize}
\item[(i)] There exists $f:[t_0,t_1]\rightarrow X$ such that
$$\lim_{n\rightarrow\infty}\|f_n(t)-f(t)\|_X=0,$$
for almost every $t\in[t_0,t_1]$.\vspace*{0.2cm}
\item[(ii)] There exists a Lebesgue integrable function $g:[t_0,t_1]\rightarrow \mathbb{R}$ such that
$$\lim_{n\rightarrow\infty}|g_n(t)-g(t)|=0,$$
for almost every $t\in[t_0,t_1]$.\vspace*{0.2cm}
\item[(iii)] $\|f_n(t)\|_X<g_n(t)$, for almost every $t\in[t_0,t_1]$.\vspace*{0.2cm}
\item[(iv)] $\lim_{n\rightarrow\infty}\int_{t_0}^{t_1}g_n(s)\,ds=\int_{t_0}^{t_1}g(s)\,ds<\infty$.\vspace*{0.2cm}
\end{itemize}
Then $f(t)$ is Bochner integrable and we have
$$\lim_{n\rightarrow\infty}\int_{t_0}^{t_1}\|f_n(s)-f(s)\|_X\,ds=0.$$
In particular we have
$$\lim_{n\rightarrow\infty}\int_{t_0}^{t_1}f_n(s)\,ds=\int_{t_0}^{t_1}f(s)\,ds,$$
in the topology of $X$.
\end{theorem}

Another classical result that need to be stated is the differentiation under the integral sign to Bochner integrable functions. It worths to point that this result can be obtained as a consequence of Hille's Theorem (see \cite[Theorem 1.2.2]{Vr1} for details on the proof of Hille's Theorem).

\begin{theorem}[Differentiation under the integral sign]\label{leibnizint} Let $f:[t_0,t_1]\times[t_0',t_1']\rightarrow X$ be a Bochner measurable function, which is Bochner integrable with respect to the second variable and $\phi,\psi:[t_0,t_1]\rightarrow[t_0',t_1']$ be differentiable functions. If there exists $(\partial/\partial t)f(t,s)$ for almost every $(t,s)\in[t_0,t_1]\times[t_0',t_1']$ and it is Bochner integrable with respect to $s$ in $[t_0',t_1']$, then
\begin{equation*}\dfrac{d}{dt}\int_{\psi(t)}^{\phi(t)}f(t,s)\,ds=f(t,\phi(t))\phi'(t)-f(t,\psi(t))\psi'(t)+\int_{\psi(t)}^{\phi(t)}\dfrac{\partial}{\partial t}f(t,s)\,ds\end{equation*}
for almost every $t\in[t_0,t_1]$.
\end{theorem}

\end{document}